\documentclass[10pt,a4paper]{amsart}

\usepackage{hyperref}
\usepackage{mathrsfs}
\usepackage{amsmath}
\usepackage{amssymb}
\usepackage[all]{xy}
\usepackage{latexsym}
\usepackage[T1]{fontenc}
\title{On the derived functors of destabilization at odd primes}
\author[Geoffrey Powell]{Geoffrey M.L. Powell}
\address{Laboratoire Angevin de Recherche en Mathématiques, UMR 6093, 
  Faculté des Sciences, Université d'Angers, 
2 Boulevard Lavoisier,
49045 Angers, France}
\email{Geoffrey.Powell@math.cnrs.fr}
\thanks{This research was partially supported by the project ANR `HGRT', BLAN08-2 338236, HGRT}
\date{}
\newtheorem{theorem}{Theorem}[section]
\newtheorem{proposition}{Proposition}[section]
\newtheorem{corollary}{Corollary}[section]
\newtheorem{lemma}{Lemma}[section]
\newtheorem{THM}{Theorem}
\theoremstyle{definition}
\newtheorem{definition}{Definition}[section]
\newtheorem{example}{Example}[section]
\theoremstyle{remark}
\newtheorem{remark}{Remark}[section]
\newtheorem{nota}{Notation}[section]
\newcommand{\boldGamma}{\mathbf{\Gamma}}
\newcommand{\boldDelta}{\mathbf{\Delta}}
\newcommand{\rkan}{R_{\mathrm{Kan}}}
\newcommand{\bd}{^{\mathrm{bd}}}
\newcommand{\dash}{\mbox{-}}
\newcommand{\chcx}{\mathrm{Ch}}
\newcommand{\dhom}{\mathbb{D}}
\newcommand{\dcx}{\mathfrak{D}}
\newcommand{\even}{^{\mathrm{ev}}}
\newcommand{\odd}{^{\mathrm{odd}}}
\newcommand{\rfunctilde}{\widetilde{\mathscr{R}}}
\newcommand{\rfunc}{\mathscr{R}}
\newcommand{\Philocal}{\mathbf{\Phi}}
\newcommand{\oddgen}{\tilde{M}}
\newcommand{\euler}{\mathfrak{e}}
\newcommand{\sltilde}{\widetilde{SL}}
\newcommand{\largetensor}{\underline{\underline{\otimes}}}
\newcommand{\destab}{\mathsf{D}}
\newcommand{\sym}{\mathfrak{S}}
\newcommand{\alt}{\mathfrak{A}}

\newcommand{\amod}{\mathscr{M}}
\newcommand{\amodbd}{\amod\bd}
\newcommand{\unst}{\mathscr{U}}
\newcommand{\st}{\mathrm{St}}
\renewcommand{\phi}{\varphi}
\renewcommand{\epsilon}{\varepsilon}
\newcommand{\natop}{{\nat\op}}
\newcommand{\nat}{\mathbb{N}}
\newcommand{\zed}{\mathbb{Z}}
\newcommand{\field}{\mathbb{F}}
\newcommand{\obj}{\mathrm{Ob}\ }
\newcommand{\op}{^\mathrm{op}}
\newcommand{\cala}{\mathcal{A}}
\newcommand{\calc}{\mathcal{C}}
\begin{document}

\maketitle

\begin{abstract}
An explicit chain complex is constructed to calculate the derived functors of
destabilization at an odd prime,
generalizing constructions of Zarati and of Hung and Sum. The methods are based
on the ideas of Singer and Miller and also apply
at the prime two. A  structural result on the derived functors of
destabilization is deduced.
\end{abstract}

\section{Introduction}

The destabilization functor $\destab$ from the category $\amod$ of modules over
the mod $p$ Steenrod algebra $\cala$ to
the category $\unst$ of unstable modules is the left adjoint to the inclusion
$\unst \hookrightarrow
\amod$; it is right exact and has non-trivial left derived functors $\destab_s:
\amod \rightarrow \unst$.
These derived functors are of considerable interest in homotopy theory: for
example
Lannes and Zarati \cite{lannes_zarati_deriv_destab,zarati_these} used them to
prove a weak version of the Segal
conjecture; they have recently been used by Hai,  Schwartz and Nam 
\cite{hai_nam_schwartz} ($p=2$) and
Hai \cite{hai_BSMF} (for $p$ odd) to prove a generalization of the weak Segal
conjecture. This project was motivated in part by the need to gain a
better understanding of the action of Lannes' $T$-functor on the derived
functors of destabilization at odd primes.

At the prime two, Singer constructed functorial chain complexes with homology
calculating the derived functors of iterated loop functors
\cite{singer_loops_II}; these can be used to construct chain complexes for
calculating the functors $\destab_s$ (cf. Goerss \cite{goerss}, who works with
homology). Lannes and Zarati \cite{lannes_zarati_deriv_destab} gave an
independent approach at the prime two, calculating  the derived functors
$\destab_s (\Sigma^{-t}M)$ where $M$ is an unstable module and $t
\leq s$.

Both approaches  depend upon the  relationship between the Steenrod algebra and
the Dickson invariants $H^*(BV_s)^{GL_s}$ (the
algebra of invariants of the cohomology $H^* (BV_s)$ of a rank $s$ elementary
abelian $2$-group $V_s$ under the action
of the general linear group $GL_s$) for varying  $s$. For odd primes, this
relationship is more subtle and  is explained by the results of M{\`u}i
\cite{mui_mod_invt_symm} in terms of  an explicit subalgebra of the invariants
$H^* (BV_s) ^{\sltilde_s}$, where $\sltilde_s$ is an index two subgroup of
$GL_s$ containing the special linear group. 

The purpose of this paper is to give a construction of a functorial chain
complex to calculate the derived functors of
destabilization for odd primes. This unifies and  builds upon results in the
literature: the derived functors of iterated
loop functors at odd primes were studied by Li in his thesis \cite{li_these}
and, in joint work with Singer
\cite{li_singer}, he gave a chain complex for calculating the homology of the
Steenrod
algebra; Hung and Sum
\cite{hung_sum} modified and generalized to odd primes the approach of Singer
\cite{singer_invt_lambda}, leading to an
invariant-theoretic description of the Lambda algebra at odd primes; Zarati
generalized his work 
with Lannes \cite{lannes_zarati_deriv_destab} to the odd primary case.

The methods of this paper use a generalization of Zarati's constructions to all
$\cala$-modules, drawing on the
observations of Hung and Sum, which are based on ideas going back to Singer and
Miller. Writing $\chcx_{\geq 0} \amod$
for the abelian category of homological chain complexes (in non-negative
degrees), the main result is the following:

\begin{THM}
\label{THM:chain_complex}
 There is an exact functor
$
 \dcx_\bullet : \amod \rightarrow \chcx_{\geq 0} \amod
$
such that, for $M \in \obj \amod$ and $s \in \nat$, there is an isomorphism of
$\cala$-modules:
\[
 \destab_s M \cong H_ s (\dcx_\bullet M).
\]
\end{THM}

The proof that the chain complex $\dcx_\bullet M$ calculates the derived
functors of destabilization is inspired by the
argument of Singer  \cite{singer_loops_II} for the case of the derived functors
of iterated loop functors and avoids
explicit calculation by using a connectivity argument for the chain complex
(similar to an argument used by Zarati
\cite{zarati_these}). This recovers the results of Zarati and the method also
applies at the prime two.

As is clear from Singer's work, underlying these methods is the fact that the
Steenrod algebra is a quadratic algebra
which is non-homogeneous Koszul (in the terminology of Priddy \cite{priddy}).
The destabilization functor appears by
exploiting the total Steenrod power, thus going back to the very origins of
instability.

Structural results on the derived functors of destabilization can be deduced
using these methods.
Recall that the Dickson invariants (the polynomial part of $H^*
(BV_s) ^{GL_s} $ for $p$ odd) identifies as the polynomial algebra
$\field_p [Q_{s,0}, \ldots , Q_{s,s-1}]$ and that, if $\Phi : \amod \rightarrow
\amod$ denotes the Frobenius functor on
the category  $\amod$, then $\Phi^k \field_p [Q_{s,0}, \ldots , Q_{s,s-1}]$
identifies with the subalgebra
$\field_p [Q_{s,0}^{p^k}, \ldots , Q_{s,s-1}^{p^k}]$.

\begin{THM}
\label{THM:module_structures}
For $M \in \obj \unst$ and  $s, t, w\in \nat$ such that $p^w \geq \big[ 
\frac{t-s+1}{2} \big]$,
 $\destab_s (\Sigma^{-t} M)$ is equipped with a natural $\Phi^{w+1}\field_p
[Q_{s,0}, \ldots ,
Q_{s,s-1}]$-structure in  $\unst$.

If $t \leq s$,  then $\destab_s (\Sigma^{-t}M ) $ has a natural $ \field_p
[Q_{s,0}, \ldots, Q_{s,s-1}]$-module
structure in $\unst$.
\end{THM}

This result allows the very rich structure of unstable modules over Noetherian
unstable algebras to be brought to bear; analogous results were obtained for
derived functors of iterated loop functors (at the prime two) in
\cite{powell_mod_loop}.

Since the current work was made available \cite{pow_destab}, Kuhn and McCarty
\cite{arXiv:1109.3694} have given a derivation of a
chain complex (working with homology) for calculating the corresponding
functors $\Omega^\infty_s$ at the prime two. Their construction of the chain
complex uses the topological origin of the Singer functors, thus relying on
classical calculations of the homology of infinite loop spaces in terms of the
Dyer-Lashof algebra to avoid the algebraic input used in this paper (but for
$p=2$). The proof that the chain complex calculates the derived functors of
destabilization follows the strategy of this work.

\bigskip
{\bf Outline of the proof:}
The construction of the chain complex begins with the definition of an exact
functor $\rfunc_1 :
\amod \rightarrow \amod$, given explicitly as a subfunctor of $M \mapsto H^*
(BV_1) [Q_{1,0}^{-1}] \largetensor M$,
where  $\largetensor $ is a half-completed tensor product. This is achieved by
first constructing a larger functor
$\rfunctilde_1$ and then restricting using an eigenspace decomposition
associated to an action
by the group $\zed/2$. The functor $\rfunc_1$ is a generalization of the functor
$R_1$ defined by Zarati \cite{zarati_these} on the category of unstable modules;
there is
a technical difficulty introduced by passage to the category of $\cala$-modules,
since the functor $\rfunc_1$ does not commute with inverse limits.

There are higher functors $\rfunc_s$, $s \in \nat$, which are constructed by
iteration and using a restriction induced by the definition of $\rfunc_2$; this
corresponds to the quadratic nature of $\rfunc_\bullet$.
Namely, $\rfunc_2 \subset \rfunc_1 \rfunc_1$ is
constructed and,
for $s \geq 2$, the functor $\rfunc_s$ can be defined as
$\bigcap_{i+j = s-2} \rfunc_1^{\circ i} \circ \rfunc_2 \circ \rfunc_1 ^{\circ
j}$ as subfunctors of $\rfunc_1^{\circ s}$. To compare with the work of Zarati
\cite{zarati_these} and Hung and Sum \cite{hung_sum}, the functors $\rfunc_s$
are identified as subfunctors of $M \mapsto H^* (BV_s) [Q_{s,0}^{-1}]
\largetensor M$; the presence of the half-completed tensor product involves a
number of unavoidable technicalities. 

The differential of the chain complex is induced by Singer's $\cala$-linear
morphism $H^* (BV_1) [Q_{1,0}^{-1}]  \rightarrow \Sigma^{-1} \field_p$, which
gives  a natural transformation $\rfunc_1 M \rightarrow \Sigma^{-1} M$. The
definition of  $\rfunc_1 M$ ensures that  the cokernel of this morphism is the
module $\Sigma^{-1} \destab M$. The higher differentials are induced by using
the canonical inclusions $\rfunc_s \subset \rfunc_{s-1} \rfunc_1$, as the
composites
\[
 \rfunc_s (\Sigma ^{s-1} M)
\hookrightarrow 
\rfunc_ {s-1} \rfunc_1 \Sigma^{s-1} M
\stackrel{\rfunc_1 d_{\Sigma^{s-1}M}}{\longrightarrow}
\rfunc_{s-1} \Sigma^{s-1}M, 
\]
giving rise to a chain complex
\[
 \ldots
\rightarrow
\Sigma \rfunc_s \Sigma^{s-1} M
\rightarrow
\ldots
\rightarrow
\Sigma \rfunc_2 \Sigma M
\rightarrow
\Sigma \rfunc_1 M
\rightarrow
M.
\]
The proof that this is a chain complex, reduces to showing that the composite
$\Sigma \rfunc_2 \Sigma M
\rightarrow
\Sigma \rfunc_1 M
\rightarrow
M$ is trivial; this is a consequence of the relationship between the Steenrod
algebra and invariant theory, highlighted (at the prime $2$) by Singer
\cite{singer_invt_lambda}.

A formal argument is used to show that the   homology of the chain complex
calculates the derived functors of destabilization, reflecting the behaviour of
the chain complex with respect to suspension. The essential tool is the natural
transformation
\[
 \rho_1 : \rfunc_1 M \rightarrow \Sigma^{-2} \Phi \Sigma M,
\]
introduced by Zarati in the unstable module setting. This has higher analogues,
which induce a morphism of  chain complexes. Using this, the proof
proceeds by showing vanishing of the higher homology on the projective
generators
of the category $\amod$. This uses a connectivity argument, reminiscent of that
used by Zarati.

\bigskip{\bf Organization of the paper:}
Sections \ref{sect:prelim}, \ref{sect:weakcts}, \ref{sect:large} provide
background and some technical tools, in particular the notions of weak
continuity and connectivity which palliate the non-continuity of the functors
considered. The construction of the functors $\rfunc_s$ is carried out in
Sections \ref{sect:rfunc1}, \ref{sect:rfunctors}, \ref{sect:ks}, with the chain
complex appearing in Section \ref{sect:chain_complex}. The short exact sequence
of chain complexes is introduced in Section \ref{sect:ses} and these ingredients
are put together in Section \ref{sect:proof} to prove Theorem
\ref{THM:chain_complex}. Theorem \ref{THM:module_structures} is proved in
Section \ref{sect:module}.

\section{Preliminaries}
\label{sect:prelim}

The Steenrod algebra over a fixed odd prime $p$ is denoted $\cala$. The category
 $\amod$ of graded $\cala$-modules
contains the category $\unst$ of unstable modules as a full subcategory;  note
that, whereas $\cala$-modules are $\zed$-graded, the instability condition
implies that unstable modules are $\nat$-graded. 
 The categories $\amod$ and $\unst$ are abelian, complete and cocomplete (see
\cite{schwartz_book}). 
 
\subsection{The destabilization functor}
\label{subsect:destab_functor}

The destabilization functor $\destab : \amod \rightarrow \unst$ is the left
adjoint to the inclusion $\unst
\hookrightarrow \amod $; it is right exact and admits left derived functors
$\destab_ s$, $s \geq
0$. The functor $\destab $ is described explicitly as follows: for $M$  an
$\cala$-module,  the subspace $BM : = \langle \beta ^\epsilon P^i x \ |\
\epsilon +
2i > |x|, \  \epsilon \in \{0,1 \}\rangle $ is stable under the action of the
Steenrod algebra and $\destab M \cong M / BM$.

\subsection{Algebras and modules}

The categories $\amod$ and $\unst$ are symmetric monoidal with respect to the
tensor product of graded vector spaces and symmetry using the Koszul sign
convention. Unital commutative algebras in $\amod$ with respect to this
structure are referred to here simply as algebras; an  algebra $K$ in $\amod$ is
unstable if the underlying $\cala$-module is unstable and the Cartan condition
is satisfied (if $k$ is an element of even degree, $P^{|k|/2} k = k^p$). For $B$
an algebra in $\amod$, 
a $B$-module in $\amod$ is an $\cala$-module equipped with a $B$-module
structure
such that the structure morphism is  $\cala$-linear; the category of $B$-modules
in $\amod$ is denoted by $B\dash \amod$. Similarly, if $K$ is an unstable
algebra, the category of $K$-modules in $\unst$ is denoted by $K\dash \unst$, so
there is a forgetful functor $K\dash\unst \rightarrow K\dash\amod$. These
categories are both abelian.

\subsection{The Frobenius functor, $\Phi$}

The exact functor $\Phi : \amod \rightarrow \amod$ is the analogue for odd
primes of the familiar doubling functor for
$p=2$:
\[
(\Phi M)^n =
\left \{
\begin{array}{ll}
M^{2k} & n = 2 k p \\
M^{2k+1} & n = 2kp + 2\\
0 & \mathrm{otherwise.}
\end{array}
\right.
\]
In particular, $\Phi M $ is concentrated in even degrees and an element $x \in
M^d$ gives rise to $\Phi x \in M^n$,
where $n = pd$ if $d$ is even and $n= p (d-1) +2 $ if $d$ is odd. The action of
the Steenrod algebra is given by
\[
\begin{array}{lllll}
  P^i (\Phi x) &= &\Phi (P^{i/p} x) & p | i , & |x | \equiv 0 (2)  \\
P^i (\Phi x)&  =& \Phi (\beta P^{(i-1)/p} x) & p |i -1, & |x | \equiv 1 (2) \\
\beta (\Phi x)& =& 0.
\end{array}
\]

For $M \in \obj \amod$, there is a natural $\field_p$-linear morphism
$
\lambda_M :
\Phi M \rightarrow M
$
defined by $\lambda_M (\Phi x ) := \beta^\epsilon P^i x$, where $|x|= 2i +
\epsilon$ and $\epsilon \in \{0,1
\}$. If $M$ is unstable, then
$\lambda_M$ is $\cala$-linear.

\begin{remark}
For $M \in \obj \amod$, the canonical surjection $M \twoheadrightarrow
\destab M$ induces a surjection
$\destab \Phi M \twoheadrightarrow \Phi \destab M$.
For $p$ odd, this is not in general an isomorphism; for example, consider the
free unstable module $F(1)$ on a
generator of degree one.
\end{remark}

\subsection{Cohomology of elementary abelian $p$-groups and invariants}

For $V$ an elementary abelian $p$-group, $H^* (BV)$ denotes the
$\field_p$-cohomology of the classifying space
$BV$; $V_s$ will denote an elementary abelian $p$-group
of rank $s\in \nat$. There is an isomorphism
of unstable  algebras $H^* (BV_1) \cong \Lambda (x) \otimes \field_p [y]$, where
$\Lambda (-)$ denotes the exterior
algebra functor, $|x| = 1 $,  $|y| = 2$ and the generators are linked by the
Bockstein operator $\beta x =y$. There is a Künneth
isomorphism
$
H^* (BV_s )
\cong
\Lambda (x_1, \ldots , x_s ) \otimes \field_ p [y_1, \ldots , y_s]
$
 and  $\field_p[y_1, \ldots, y_s]  \subset H^* (BV_s)$ is a sub unstable
algebra.

The general linear group $GL_s := GL(V_s) $ acts naturally on $H^* (BV_s)$ by
morphisms of unstable algebras and
$\field_p[y_1, \ldots , y_s]$ is stable under this action. In particular, for 
$G \subset GL _s$,  the algebras of invariants 
$ 
\field _p [y_1, \ldots , y_s] ^G \hookrightarrow H^* (BV_s)^G
\hookrightarrow
H^* (BV_s)
$ 
are unstable algebras.

\begin{remark}
Below and in Section \ref{subsect:mui}, the classical approach to the
Dickson-M{\`u}i invariants is reviewed for the
convenience of the reader, who may also wish to consult the paper by Kameko and
Mimura \cite{kameko_mimura}, which gives a presentation using cohomology
operations.
\end{remark}

The Dickson invariants $\field _p [y_1, \ldots , y_s] ^{GL_s}$ form a polynomial
algebra 
 $$
\field_p [Q_{s,0}, Q_{s, 1},\ldots, Q_{s, s-1} ],
$$
where $|Q_{s,i}|=2 ( p^s -p^i)$. The generators $Q_{s,i}$ are defined by
\[
f_s (X) := \prod_{y \in V_s^*} (X - y)
= \sum_{i=0}^s
(-1)^{s-i}Q_{s,i} X^{p^i}.
\]
In particular, the top Dickson invariant, $Q_{s,0}$, is $ (-1)^s \prod_{y \in
V_s ^* \backslash 0 } y$.

\begin{nota}
	For $0< s \in \nat$, the Dickson invariants $\field _p [y_1, \ldots ,
y_s] ^{GL_s}$ will be abbreviated to $\field_p [Q_{s,i}]$.
\end{nota}

\begin{definition}
(Cf.  \cite{mui_cohom_operations}.)
	Let $\sltilde_s \subset GL_s$ denote the kernel of the morphism $GL_s
\twoheadrightarrow \zed/2$,
$g \mapsto  \det(g)^{\frac{p-1}{2}}$.
\end{definition}

The  morphism $GL_s \twoheadrightarrow \zed/2$ is a split surjection, leading
to:

\begin{lemma}
\label{lem:general_eigenspace}
Suppose $GL_s$ acts by morphisms of
unstable algebras on $K$. Then
\begin{enumerate}
 \item 
 $K^{\sltilde_s} \subset K$ inherits a  canonical $\zed/2$-action  and 
$(K^{\sltilde_s}) ^{\zed/2}\cong K^{GL_s}$;
\item
there
is a canonical splitting $
K ^{\sltilde_s} \cong K ^{GL_s} \oplus (K ^{\sltilde_s})^-
$ in  $K^{GL_s}\dash \unst$.
\end{enumerate}
\end{lemma}

\subsection{M{\`u}i invariants}
\label{subsect:mui}

Fix an ordered basis $\{ x_1, \ldots, x_s \}$ of $V_s^* \cong H^1 (BV_s)$,
giving an
ordered set $\{ y_1, \ldots, y_s \}$ of generators of the polynomial part of
$H^* (BV_s)$, where $\beta x_i = y_i$.

\begin{definition}
\label{def:invariants}
(Cf. \cite{mui_cohom_operations,mui_mod_invt_symm,hung_sum}.)
For integers $0 \leq i <s $, define:
\begin{enumerate}
 \item 
 \[L_s := 
 \left|
\begin{array}{ccc}
y_1 &\cdots&y_s \\
y_1^p & \cdots & y_s^p \\
\vdots & \cdots & \vdots \\
y_1 ^{p^{s-1} } & \cdots &y_s^{p^{s-1} }
\end{array}
\right|\] 
and 
 $\euler_s := L_s ^{\frac{p-1}{2}} $ of degrees $|L_s|=2
\Big(\frac{p^s -1} {p-1} \Big)$ and $|\euler_s|= p^s-1$;
 \item
 \[M_{s,i}:= 
 \left|
\begin{array}{cccc}
x_1 &\cdots& x_s  \\
y_1 & \cdots & y_s  \\
\vdots & \cdots & \vdots \\
y_1 ^{p^{s-1}} & \cdots & y_s^{p^{s-1}}
\end{array}
\right|,
\]
in which the row $\big( y_1 ^{p^i} \  \cdots \  y_s ^{p^i}\big)$ is omitted from
the array;
\item
$
\oddgen_{s,i} := M_{s,i} L_s ^{\frac{p-3}{2}}
$
of degree $|\oddgen_{s,i}|= p^s - 2 p^i$;
\item
$
		R_{s,i} := \oddgen_{s,i} \euler_s
	$
	of degree $|R_{s,i}|= 2 (p^s - p^i) -1$. 
\end{enumerate}
\end{definition}

\begin{proposition}
\cite{mui_cohom_operations}
\label{prop:GL_sltilde_invariance}
For integers  $0 \leq i < s$, 
\begin{enumerate}
\item
$L_s \in H^* (BV_s) ^{SL_s} $;
\item
$\euler_s, \oddgen_{s,i} \in H^* (BV_s) ^{\sltilde_s} $;
\item
$Q_{s,0}, R_{s,i} \in H^* (BV_s) ^{GL_s} $.
\end{enumerate}
Moreover, $\euler_s^2 = Q_{s,0}$.
\end{proposition}

\begin{example}
\label{exam:invariants_s=1}
There are identifications
\[
H^* (BV_1) ^{GL_1} \cong
\Lambda (\oddgen_{1, 0} \euler_1 ) \otimes \field_p [Q_{1,0}]
\hookrightarrow
\Lambda (\oddgen_{1,0}) \otimes \field_p [\euler_1]
\cong  H^* (BV_1) ^{\sltilde_1} ,
\]
where $\euler_1 = y^{\frac{p-1}{2}}$, $\oddgen_{1,0} = x y^{\frac{p-3}{2}}$ and
 $Q_{1,0}=  y^{p-1}$. Moreover, $(H^* (BV_1) ^{\sltilde_1})^-$ is
the free $\field_p [Q_{1,0}]$-module on generators $\oddgen_{1,0}$ and
$\euler_1$.
\end{example}

\subsection{Localization}

The interest of localizing $H^* (BV_s)$ by inverting the top Dickson invariant
$Q_{s,0}$ is well-established (cf.
\cite{wilkerson,singer_localization_modules}).

\begin{nota}
Denote by  $T_s \subset GL_s$ the $p$-Sylow subgroup of upper
triangular matrices with respect to the chosen ordered basis of $V_s^*$.
\end{nota}

Thus there are inclusions $T_s \subset SL_s \subset \sltilde_s \subset
GL_s$. The following algebras in $\amod$ play an important role (cf.
\cite{hung_sum}).

\begin{nota}
For $0< s\in \nat$, 
\begin{enumerate}
 \item 
 $\Philocal_s:= H^* (BV_s) [Q_{s,0}^{-1}]$;
\item
$\boldGamma_s:= \Philocal_s ^{GL_s} \cong (H^* (BV_s)
^{GL_s} ) [Q_{s,0}^{-1}]$;
\item
$\boldDelta_s:=\Philocal_s ^{T_s} \cong (H^* (BV_s)
^{T_s} ) [Q_{s,0}^{-1}]$.
\end{enumerate}
\end{nota}

There are inclusions of algebras in $\amod$: $\boldGamma_s \subset \boldDelta_s
 \subset \Philocal_s$. Clearly $\boldDelta_1 = \Philocal_1$.

\begin{proposition}
\label{prop:Gamma_Delta}
 For $0<s \in \nat$, there are isomorphisms of algebras
\begin{enumerate}
\item
 \cite[Corollary 1.2(ii)]{hung_sum}
$\boldGamma_s \cong \Lambda (R_{s,i}| 0 \leq i \leq s-1) \otimes \field_p
[Q_{s,i}| 0 \leq i \leq s-1][Q_{s,0}^{-1}]$;
	\item
	 \cite[Corollary 1.3]{hung_sum}
	 $\boldDelta_s \cong
\Lambda (u_1, \ldots , u_s) \otimes \field_p [v_1^{\pm 1}, \ldots , v_s^{\pm
1}]$, where
$|u_i|= 1$ and $|v_i|=2$.
\end{enumerate}
\end{proposition}

\subsection{On $\sltilde_s$-invariants} The following result introduces 
the unstable algebra which underlies the constructions of this paper and
identifies with the
algebra of $\sltilde_s$-invariants after inverting $Q_{s,0}$. Recall that the
regular representation $V_s \hookrightarrow \sym_{p^s}$ factors canonically
across the inclusion
$\alt_{p^s} \subset \sym_{p^s}$ of the alternating group in the symmetric group.

\begin{theorem}
\label{thm:sltildes_inv}
(\cite[I.4.14]{mui_mod_invt_symm} and \cite{mui_cohom_operations}.)
For $0< s \in \nat$, the elements $ \{ \oddgen_{s,i}|  0 \leq i \leq s-1 \}$ and
 $\{\euler_s, Q_{s,j} | 1 \leq j \leq s-1 \}$ generate a free graded commutative
algebra
\[
\Lambda (\oddgen_{s,i} ) \otimes \field _p [\euler_s, Q_{s,j} ] \subset
H^* (BV_s) ^{\sltilde_s},
\]
which is a sub unstable algebra, isomorphic to the image of the restriction
morphism
$
 H^* (B\alt_{p^s} ) \rightarrow H^* (BV_s).
$
This induces  an isomorphism
\[
 \Philocal_s ^{\sltilde_s}
\cong
\Lambda (\oddgen_{s,i}  ) \otimes \field _p [\euler_s, Q_{s,j} ][Q_{s,0}^{-1}].
\]
\end{theorem}

\section{Weak continuity and connectivity}
\label{sect:weakcts}

This section introduces technical conditions which palliate the failure of the
functors used in the 
constructions to be continuous (a functor is continuous if it commutes with
small limits).

\subsection{Weakly continuous functors}

The category of inverse systems in $\calc$ is written $\calc^\natop$.
Write $\amodbd \subset \amod$ for the full subcategory of bounded-above
modules and, for  $M \in \obj \amod$ and $c \in \zed$, let $M ^{\geq c} \subset
M$ denote the sub
$\cala$-module of elements of degree at least $c$ and $M^{<c }$ the quotient
module $M / M^{\geq c}$.

The following resumes standard properties:

\begin{lemma}
\label{lem:truncation}
For $M \in \obj \amod$ and $c \in \zed$, 
\begin{enumerate}
	\item
$M \mapsto M^{\geq c}$ and $M \mapsto M^{<c}$ define exact
functors $(-)^{\geq c}$ , $(-)^{<c}$ on  $\amod$ and there is  a natural short
exact sequence
$
0
\rightarrow
M^{\geq c}
\rightarrow
M
\rightarrow
M^{< c}
\rightarrow
0
$
in $\amod$.
\item
The modules $M^{<c}$ define a functorial inverse system of surjections $
 \ldots \twoheadrightarrow M^{< c +1 } \twoheadrightarrow M^{< c}
\twoheadrightarrow \ldots $
in $\amodbd$.
\item
The inverse limit induces a functor
$\lim_\leftarrow
:
(\amodbd) ^\natop
\rightarrow
\amod$
and the canonical morphism $M \rightarrow \lim_\leftarrow (M ^{<c})$ is an
isomorphism in $\amod$.
 \end{enumerate}
\end{lemma}

\begin{nota}
	\
	\begin{enumerate}
		\item
		For $\alpha : \amod \rightarrow \amod$ a functor, let $\alpha
\bd : \amodbd \rightarrow \amod$ denote
the restriction of $\alpha$ to $\amodbd$.
\item
For $\overline{\beta}: \amodbd \rightarrow \amod$ a functor, let $\rkan
\overline{\beta}: \amod \rightarrow \amod$
denote the right Kan extension of $\overline{\beta}$, given by
$
	\rkan \overline{\beta} (M) := \lim_\leftarrow \overline{\beta} (M^{<c})
.$
	\end{enumerate}
\end{nota}

\begin{definition}
	A functor $\alpha : \amod \rightarrow \amod$ is weakly continuous if the
canonical natural transformation
	 $
		\alpha \rightarrow \rkan (\alpha \bd)
	 $
	is an isomorphism.
\end{definition}

\begin{example}
 A continuous functor $\alpha : \amod \rightarrow \amod$ is clearly weakly
continuous; the basic example of a weakly continuous functor which is not
continuous arises from the large tensor product reviewed in Section
\ref{sect:large} (see Example \ref{exam:weak_not_strong}).
\end{example}

\begin{proposition}
\label{prop:properties_weak_cts}
	Let $\overline{\alpha}, \overline{\beta}, \overline{\gamma}$ be functors
$\amodbd \rightarrow \amod$. Then
	\begin{enumerate}
		\item
if $\overline{\alpha}$ is exact,  $\rkan \overline{\alpha}$ is exact;
\item
if $0 \rightarrow \overline{\alpha}\rightarrow \overline{\beta}\rightarrow 
\overline{\gamma} \rightarrow 0$ is a short
exact sequence of exact functors, then
$$
0 \rightarrow \rkan \overline{\alpha}\rightarrow 
\rkan\overline{\beta}\rightarrow   \rkan\overline{\gamma} \rightarrow
0
$$
		is a short exact sequence of functors;
		\item
		there is a bijection of sets of natural transformations:
		\[
			\mathrm{Nat} (\overline{\alpha}, \overline{\beta} )
			\cong
			\mathrm{Nat} (\rkan \overline{\alpha},
\rkan\overline{\beta} ).
		\]
	\end{enumerate}
\end{proposition}

\begin{proof}
	If $0 \rightarrow A \rightarrow B \rightarrow C\rightarrow 0$ is a short
exact sequence of $\cala$-modules, then
there is an induced short exact sequence
$
	0 \rightarrow A^{<\bullet} \rightarrow B^{<\bullet} \rightarrow
C^{<\bullet}\rightarrow 0
$
of inverse systems of bounded above $\cala$-modules in which the transition
morphisms of the inverse systems are
surjective. Applying $\overline{\alpha}$ gives a short exact sequence of inverse
systems
$
	0 \rightarrow
	\overline{\alpha} (A^{<\bullet}) \rightarrow
	\overline{\alpha} (B^{<\bullet}) \rightarrow
	\overline{\alpha} (C^{<\bullet})\rightarrow 0
$ 
such that the transition morphisms are surjective in each of the inverse
systems, since $\overline{\alpha}$ is exact by
hypothesis. The first statement follows by Mittag-Leffler.

The proof of the second statement is similar and the final statement follows
 from the properties of the right Kan extension.
\end{proof}

\subsection{Stable inverse systems}

Since a weakly continuous functor need not be continuous, it is useful to
place a restriction on the
inverse systems which are considered.

\begin{definition}
	An inverse system $X _\bullet$ of $\amod^\natop$ is stable if there
exists a map $\mu : \nat
\rightarrow \nat$ such that
	\begin{enumerate}
	\item
		$\lim _{s \rightarrow \infty } \mu (s) = \infty$.
	\item
		$X_{s+1}^{<  \mu(s)} \rightarrow X_s ^{< \mu(s)} $ is an
isomorphism, $\forall s \in \nat$;

	\end{enumerate}
\end{definition}

\begin{lemma}
\label{lem:stable_weak_continuity}
	Let $\alpha : \amod \rightarrow \amod$ be a weakly continuous functor,
$X_\bullet $ be a stable inverse system
of $\cala$-modules and write $X _\infty$ for $\lim_\leftarrow X_\bullet$. Then
the natural morphism
	$
		\alpha (X_\infty)
		\rightarrow
		\lim _{\leftarrow i} \alpha (X_i)
	$
is an isomorphism.
\end{lemma}

\begin{proof}
	Since $\alpha$ is weakly continuous,  $\alpha(X_i) \cong
\lim_{\leftarrow c}
\alpha(X_i^{<c}) $, hence 
	\[
		\lim_{\leftarrow i} \alpha (X_i)
		\cong
		\lim _{\leftarrow i} \lim_{\leftarrow c} \alpha (X_i^{<c})
		\cong
		\lim _{\leftarrow c} \lim_{\leftarrow i} \alpha (X_i^{<c}),
	\]
by reversing the order of the inverse limits. For fixed $c$, the stability
hypothesis  implies
that $X_i ^{<c}$ is isomorphic to $X_\infty ^{<c}$ for $i \gg 0$, hence $
\lim_{\leftarrow i} \alpha (X_i^{<c}) \cong
\alpha (X_\infty^{<c})$. Finally, since $\alpha$ is weakly continuous,   $\alpha
(X_\infty) \cong \lim_{\leftarrow  c} \alpha (X_\infty^{<c})$.
\end{proof}

\subsection{Weak connectivity for functors}
As usual, write $|M| := \mathrm{inf} \{ d
| M_d\neq 0 \}  \in \zed \cup \{ - \infty , \infty \}$ for 
the  connectivity of $M \in \obj \amod$.

\begin{definition}
	A functor $\alpha: \amod \rightarrow \amod$ is weakly connective if
there exists a function $\kappa : \zed
\rightarrow \zed$ such that
	\begin{enumerate}
	\item
	$\lim_{t \rightarrow \infty}
	\kappa (t)= \infty$;
		\item
	 $|\alpha (M) | \geq \kappa (|M|)$ if $|M|\in \zed$.	
	\end{enumerate}
If the above conditions are satisfied, $\alpha$ is said to be
$\kappa$-connective 
\end{definition}

\begin{lemma}
\label{lem:kappa-connectivity_converge}
	Let $\alpha : \amod \rightarrow \amod$ be an exact, $\kappa$-connective
functor. Then the natural morphism
	$
		\alpha (M) ^{<\kappa (c) }
		\rightarrow
		\alpha (M^{<c} ) ^{< \kappa (c)}
	$
is an isomorphism; in particular, the inverse system $\alpha(M^{< \bullet})$ is
stable.
\end{lemma}

\begin{proof}
	Straightforward.
\end{proof}

\begin{proposition}
\label{prop:composition_weakly_cts_connective}
	Let $\beta, \gamma$ be weakly continuous functors, such that $\gamma$ is
exact and $\kappa$-connective. Then the
composite $\beta \circ \gamma$ is weakly continuous. Moreover
	\begin{enumerate}
		\item
		if $\beta$ is exact, then $\beta \circ \gamma$ is exact;
		\item
		if $\beta$ is $\lambda$-connective, then $\beta \circ \gamma $
is $\lambda \circ \kappa$-connective.
	\end{enumerate}
\end{proposition}

\begin{proof}
Consider an $\cala$-module $M$. The hypothesis that $\gamma$ is weakly
continuous implies that $\gamma(M) \cong
\lim_{\leftarrow  c} \gamma (M^{<c})$ and  Lemma
\ref{lem:kappa-connectivity_converge} implies that $\gamma (M^{<\bullet}) $ is
stable.
Thus $\beta \circ \gamma (M) \cong  \beta ( \lim_{\leftarrow  c} \gamma
(M^{<c}))$ and, by Lemma 
\ref{lem:stable_weak_continuity}, the right hand side is isomorphic to $\lim
_{\leftarrow  c} \beta \circ \gamma
(M^{<c})$, since $\beta$ is weakly continuous. Hence, $\beta \circ \gamma$ is
weakly continuous; moreover, if $\beta$ is exact, then so is  $\beta \circ
\gamma$. The  final statement on the connectivity is clear.
\end{proof}

Examples of weakly connective functors are provided by the derived functors of
destabilization (see Corollary \ref{cor:stronger_connectivity}).

\section{Large tensor products}
\label{sect:large}

As in the work of Singer (eg. \cite{singer_new_chain_cx}), Lin and Singer
\cite{li_singer} and Hung and Sum \cite{hung_sum}, it is necessary to use large
tensor products when working with arbitrary modules over the Steenrod algebra.

\subsection{Large tensor products}

\begin{definition}
For $M, N  \in \obj \amod$, define $M
\largetensor N:= 
	  \lim_{\leftarrow c} (M \otimes N ^{<c}).
$
\end{definition}

\begin{proposition}
\label{prop:largetensor_properties}
	Let $M, N$ be $\cala$-modules.
\begin{enumerate}
	\item
There is a canonical embedding
$
  M \otimes N
\hookrightarrow
M \largetensor N,
$
which is an isomorphism if $N$ is bounded above or if $M$ is bounded below.
\item
The association $N \mapsto M \largetensor N$ defines an exact functor
$
	M \largetensor - : \amod \rightarrow \amod,
$
which identifies with $\rkan \Big((M \otimes - ) \bd \Big )$, in particular is
weakly continuous.
\item
The association $M \mapsto M \largetensor N$ defines an exact functor
$
	- \largetensor N : \amod \rightarrow \amod.
$
\item
The direct system $N_{\geq \bullet}$ induces an isomorphism
$
\lim_{d\rightarrow -\infty} M \largetensor (N_{\geq d})
\stackrel{\cong}{\rightarrow}
M \largetensor N.
$

\end{enumerate}
\end{proposition}

\begin{proof}
The first statement is clear, as is the identification of $M \largetensor -$
with the right Kan extension of $(M \otimes -)|_{\amodbd}$; exactness follows
from Proposition \ref{prop:properties_weak_cts}. Similarly,  the third statement
follows from the argument of Proposition \ref{prop:properties_weak_cts}.

The final statement is straightforward.
\end{proof}

\begin{example}
\label{exam:weak_not_strong}
The functor $\Philocal_1 \largetensor -$ is weakly continuous but not
continuous.
\end{example}

\subsection{Multiplicative structures}

For $M, N \in \obj \amod$, there is a natural surjection:
\begin{eqnarray}
\label{eqn:HBV_tensor}
\xymatrix{
\big( H^* (BV_s) \otimes M \big) \otimes \big(H^* (BV_s) \otimes N\big)
\ar[rr]
&&
H^* (BV_s) \otimes (M \otimes N),
}
\end{eqnarray}
induced by passage to the tensor product in $H^* (BV_s)\dash \unst$.

\begin{proposition}
\label{prop:bullet-product}
For $M, N \in \obj \amod$ and $s\in \nat$, there is a product morphism
\[
(\Philocal_s \largetensor M) \otimes (\Philocal_s \largetensor N)
\stackrel{\bullet}{\rightarrow}
\Philocal_s \largetensor (M \otimes N)
\]
in $\amod$ which extends (\ref{eqn:HBV_tensor}).
Hence,
\begin{enumerate}
\item
$\Philocal_s \largetensor M$ is a $\Philocal_s$-module in $\amod$;
\item
if $B$ is an algebra in $\amod$, then $\Philocal_s \largetensor B$ is a
$(\Philocal_s \otimes B)$-algebra in $\amod$;
\item
if, furthermore, $N \in \obj B\dash\amod$, then $\Philocal_s \largetensor
N \in \obj \Philocal_s \largetensor
B \dash\amod$.
\end{enumerate}
\end{proposition}

\begin{proof}
The result is clear with $\otimes$ in place of $\largetensor$; it is necessary
to verify that  this passes to
$\largetensor$. Using the isomorphism $\lim_{d \rightarrow -\infty}X
\largetensor ( Y_{\geq d}) \cong X \largetensor Y$
of Proposition \ref{prop:largetensor_properties},  it is sufficient to
consider the case where $M,N$ are
bounded below, $M \cong M_{\geq m}$, $N \cong N _{\geq n}$.
The surjection $M \otimes N \twoheadrightarrow ( M \otimes N)^{<c}$ factors
canonically across $M \otimes N
\twoheadrightarrow M^{< c -n} \otimes N ^{< c -m}$.
This induces a natural morphism
\[
 (\Philocal_s \largetensor M) \otimes (\Philocal_s \largetensor N)
\stackrel{}{\rightarrow}
\Philocal_s \otimes(M \otimes N)^{<c}
\]
 to the inverse system defining $\Philocal_s \largetensor (M \otimes N)$, which
is the required morphism.

The remaining statements are proved by establishing associativity and graded
commutativity, which follow from the
uncompleted case, by the construction above.
\end{proof}

\begin{corollary}
\label{cor:Philocal_s_modules}
For $M \in \obj \amod$ and $s \in \nat$, 
$\Philocal _1 \largetensor (\Philocal_s
\largetensor M)$ is naturally a $\Philocal_1 \largetensor \Philocal_s $-module
in $\amod$.
\end{corollary}

\subsection{Embeddings} An isomorphism of vector spaces $V_1 \oplus V_s
\cong\field_p \oplus \field_p ^{\oplus s}
\stackrel{\cong}{\rightarrow} \field_p ^{\oplus s+1} \cong V_{s+1}$ induces an
isomorphism of unstable algebras over the
Steenrod algebra
$
	H^* (BV_{s+1}) \stackrel{\cong}{\rightarrow} H^* (BV_1) \otimes H^*
(BV_{s}).
$

\begin{lemma}
\label{lem:Philocal_s+1_inclusion}
For $s \in \nat$ an isomorphism  $V_1 \oplus V_s
\stackrel{\cong}{\rightarrow} V_{s+1}$ induces a unique  monomorphism of
algebras in $\amod$,
	$
		\Philocal_{s+1} \hookrightarrow \Philocal_1 \largetensor
\Philocal_s
	$
which fits into a commutative diagram of morphisms of algebras in
$\amod$:
	\[
		\xymatrix{
		H^* (BV_{s+1} )
		\ar[d]
		\ar[r]^(.4)\cong
		&
		H^* (BV_1) \otimes H^* (BV_{s})
		\ar[d]
		\\
		\Philocal_{s+1}
		\ar[r]
		&
		\Philocal_1 \largetensor \Philocal_s,
		}
	\]
in which the vertical morphisms are induced by localization.
\end{lemma}

\begin{proof}
It suffices to verify that $Q_{s+1,0}$ is invertible in $\Philocal_1
\largetensor \Philocal_s$; although this is standard, the argument is given for the convenience of the reader. 

Writing the polynomial part of $H^* (BV_{s+1})$ as $\field_p [y_1, \ldots , y_{s+1}]$ and working with respect to the
isomorphism $\field_p [y_1, \ldots , y_{s+1}] \cong \field_p [y_1] \otimes \field_p [y_2, \ldots y_{s+1}]$, the Dickson
invariant can be written as
\[
	Q_{s+1,0} = -
	Q_{s,0} \prod_{\lambda \in \field_p^\times, y} (\lambda y_1 + y)
\]
where $Q_{s,0}$ is the top Dickson invariant for $\field_p [y_2, \ldots y_{s+1}]$ and $y$ ranges over the elements of
the vector space $\langle y_2, \ldots , y_{s+1} \rangle$. Hence, by reindexing and using $\prod_{\lambda \in \field_p^\times} \lambda = -1$:
\[
	Q_{s+1,0} =
	Q_{s,0} y_1^{p^s(p-1)}  \prod_{y} \Big( 1 + \frac{y}{y_1}\Big)^{p-1} \in \Philocal_1 \otimes \Philocal_s.
\]
Since 
$Q_{s,0}$ is invertible in $\Philocal_s$, it suffices to observe that each element $ 1 + \frac{y}{y_1}$ is invertible in $\Philocal_1 \largetensor \Philocal_s$; the inverse is given by
$\sum_{i\geq 0} (-1) ^i \Big(\frac{y}{v_1} \Big)^i$.
\end{proof}

By Corollary \ref{cor:Philocal_s_modules}, for $M \in \obj \amod$, $\Philocal_1
\largetensor (\Philocal_s \largetensor M)$ has a
$\Philocal_1 \largetensor \Philocal_s $-module structure, hence a
$\Philocal_{s+1}$-module structure by restriction along the monomorphism
provided by Lemma \ref{lem:Philocal_s+1_inclusion} above.

\begin{proposition}
\label{prop:otimes_embedding}
For $M \in \obj \amod$ and $s\in \nat$, an isomorphism  $V_1 \oplus V_s
\stackrel{\cong}{\rightarrow} V_{s+1}$ induces a natural
monomorphism 
$
	\Philocal_{s+1} \otimes M
	\hookrightarrow
	\Philocal_1 \largetensor (\Philocal_s \largetensor M)
$
in $\Philocal_{s+1}\dash \amod$.
\end{proposition}

\begin{proof}
	The embedding $M
	\hookrightarrow
	\Philocal_1 \largetensor (\Philocal_s \largetensor M) $ induced by the
respective units of $\Philocal_1$,
$\Philocal_s$  extends to an $\cala$-linear embedding of
$\Philocal_{s+1}$-modules, by Lemma \ref{lem:Philocal_s+1_inclusion}.
\end{proof}

\begin{remark}
	The large tensor
product $\Philocal_{s+1} \largetensor
M$ does not in general embed in $\Philocal_1 \largetensor \Big( \Philocal_s
\largetensor M
\Big)$; this stems from the fact that the functor $\Philocal_1$ is only weakly
continuous and not continuous. 
\end{remark}

\section{The functors $\rfunctilde_1$ and $\rfunc_1$}
\label{sect:rfunc1}

This section introduces the functor $\rfunc_1$ on the category $\amod$, which
generalizes the functor $R_1$ defined by
Zarati \cite{zarati_these} for $\unst$. This is achieved by first introducing a
functor
$\rfunctilde_1$ and then restricting to
$GL_1$-invariants. The non-trivial step in the construction is Proposition
\ref{prop:rfunctilde_astable}, which establishes invariance under the action of
the Steenrod algebra.

\subsection{The total Steenrod power}
\label{subsect:totalst}

The total Steenrod power $\st_1$ is fundamental to the constructions of this
paper, as
in the work of
Zarati \cite{zarati_these} and the work of Hung and Sum \cite{hung_sum}, who use
a stable version, $S_1$. The precise
relationship between the total Steenrod power and algebras of invariants was
established by M{\`u}i
\cite{mui_cohom_operations}.

\begin{nota}
Write $w$ for the element $\frac{u}{v} \in \Philocal_1  \cong
\Lambda (u)
\otimes \field_p [v^{\pm 1}]$ of
degree $-1$.
\end{nota}

\begin{definition}
(Cf. \cite[Definition 2.4]{hung_sum}.)
\label{def:S1}
For $M \in \obj \amod$, let $S_1 :  M \rightarrow \Philocal_1
\largetensor M$ be the linear morphism
\[
S_ 1 (x) :=
\sum _{i \geq  0, \epsilon \in \{0,1 \}}
(-1) ^{i+ \epsilon} w^ \epsilon Q_{1,0}^{-i}
\otimes \beta ^{\epsilon} P^i (x).
\]
For $x \in M^{2k+ \delta}$ ($\delta \in \{ 0, 1 \}$), following \cite[Section
2.4.1]{zarati_these}, define
\[
\st_1 (x): = (-1)^k \euler_1^{|x|}   S_1 (x)
\]
so that $|\st_1 (x)|= p |x|$.
\end{definition}

\begin{remark}
	The element $\st_1 (x)$ depends only upon the class $\euler_1$, hence is
 independent (up to sign) of the choice of $u,v$.
\end{remark}

The main properties of $S_1$ are resumed by the following, from which the
corresponding results for $\st_1$ can be deduced (cf. \cite{zarati_these}).

\begin{proposition}
	\label{prop:S1_properties}
\ 
\begin{enumerate}
	\item
\cite[Remark 2.11]{hung_sum}
The morphism $S_1$ takes values in $\boldGamma_1 \largetensor M$.
\item
\cite[Corollary 2.10]{hung_sum} The morphism $S_1$ is stable.
\item
\cite[Proposition 2.6]{hung_sum}
For $M_1, M_2 \in \obj \amod$ and elements $x_1 \in M_1$, $x_2 \in M_2$,
$
S_1 (x_1 \otimes x_2) = S_1 (x_1) \bullet S_1 (x_2).
$
\end{enumerate}
\end{proposition}

\subsection{The functor $\rfunctilde_1$}
\label{subsect:rfunc1_def}

Let $K_1$ denote the unstable algebra $H^* (BV_1)^{\sltilde_1}$, which was
identified in Example \ref{exam:invariants_s=1}.

\begin{definition}
 (Cf. \cite{zarati_these}.)
For $M \in \obj \amod$, let $\rfunctilde_1 M$ denote the sub $K_1$-module of
$\Philocal_1^{\sltilde_1} \largetensor M$
generated by the elements $\st_1 (x)$, $x \in M$.
\end{definition}

There are many ways to prove the following result; for brevity, a proof using
the calculations of \cite{hung_sum} is given. 

\begin{proposition}
\label{prop:rfunctilde_astable}
 The submodule $\rfunctilde_1 M \subset \Philocal_1^{\sltilde_1} \largetensor M$
is stable under the action of the
Steenrod algebra $\cala$, hence $\rfunctilde_1 M \in  \obj K_1 \dash \amod$.
\end{proposition}

\begin{proof}
It is sufficient to show that, for any element $x$ of $M$, the elements $\beta
\st_1 (x) $ and $P^i \st_1 (x)$ ($i
\in \nat$) belong to $\rfunctilde_1 M \subset \Philocal_1 \largetensor M$. Since
 $\beta \st_1 (x)=0$,  it suffices to consider the case of the reduced
powers.

By the Cartan formula, it is
straightforward to see that it is sufficient to show that $\euler_1 ^{|x|} P^i
S_1
(x)$ belongs to $\rfunctilde_1 M$ for
any $i \in \nat$. Proposition 4.10 of \cite{hung_sum} implies that
\[
 P^i (S_1 (x))
= \sum_{\epsilon \in \{0,1\}, t\geq 0}
\left(
\begin{array}{c}
 - (p-1) t - \epsilon \\
i - pt - \epsilon
\end{array}
\right)
w^\epsilon
\euler_1 ^{2(i-t) }
S_1 (\beta^\epsilon P^t x).
\]
Writing $w$ as $\oddgen_{1,0}\euler_1 ^{-1}$, this gives
\[
 \euler_1 ^{|x|} P^i (S_1 (x))
= \sum_{\epsilon \in \{0,1\}, t\geq 0}
\left(
\begin{array}{c}
 - (p-1) t - \epsilon \\
i - pt - \epsilon
\end{array}
\right)
\oddgen_{1,0} ^\epsilon
\euler_1 ^{2(i-t) + |x|- \epsilon}
S_1 (\beta^\epsilon P^t x).
\]
 To show that the right hand side belongs to $\rfunctilde_1 M$, it is sufficient
to show that the terms with 
\[
2 (i-t) + |x| - \epsilon  < |\beta^{\epsilon}P^t x |
\]
have trivial coefficient. The above condition is equivalent to $i - pt -
\epsilon< 0$, which implies the vanishing of
the binomial coefficient, as required.
\end{proof}

\begin{proposition}
\label{prop:rfunctilde1_gen_properties}
Let $M, N$ be $\cala$-modules.
\begin{enumerate}
 \item
The underlying  $K_1$-module of $\rfunctilde_1 M$ is free on $\langle \st_1 (x)
| x \in M \rangle $.
\item
The functor $\rfunctilde_1 : \amod
\rightarrow K_1 \dash \amod$ is exact
and commutes with colimits.
\item
The functor $\rfunctilde_1 : \amod \rightarrow \amod$ is weakly continuous and
$\kappa_1$-connective, where
$\kappa_1 : n \mapsto pn$.
\item
There is a natural isomorphism
$
 \rfunctilde_1 (M \otimes N)
\cong
\rfunctilde_1 (M) \otimes _{K_1} \rfunctilde_1 (N)
$
in $K_1 \dash \amod$.
\end{enumerate}
\end{proposition}

\begin{proof}
The first statement is straightforward (cf. the argument of \cite[Proposition
3.3.4]{zarati_these}). This implies the second statement and the
$\kappa_1$-connectivity. To show that $\rfunctilde_1$ is weakly continuous, the
functor $\rfunctilde_1$ can be considered as taking values in $K_1 \dash \amod$
and the inverse
limit can be formed in this category. For $M \in \obj \amod$, consider
the morphisms 
	$
		\rfunctilde_1 M
		\twoheadrightarrow
		\rfunctilde_1 M^{<c}
		\hookrightarrow
		\Philocal_1 \otimes M^{<c}.
	$ of $K_1\dash \amod$.
Passing to the inverse limit, as $c \rightarrow \infty$, this gives morphisms of
$K_1 \dash \amod$:
\[
		\rfunctilde_1 M
		\twoheadrightarrow
		\lim_{\leftarrow c} \rfunctilde_1 M^{<c}
		\hookrightarrow
		\Philocal_1 \largetensor M
	\]
where Mittag-Leffler gives the surjection. It is clear
 that the composite coincides with the
 inclusion $\rfunctilde _1 M
\hookrightarrow \Philocal_1 \largetensor M$, hence $\rfunctilde_1 M
		\twoheadrightarrow
		\lim_{\leftarrow c} \rfunctilde_1 M^{<c}$ is an isomorphism, as
required.

The tensor product property follows from the multiplicative property of $\st_1$
deduced from Proposition \ref{prop:S1_properties} and the description of the
underlying $K_1$-module;  the isomorphism is induced
by the restriction of the product morphism $(\Philocal_1 \largetensor M) \otimes
(\Philocal_1 \largetensor N)
\stackrel{\bullet}{\rightarrow} \Philocal_1 \largetensor (M \otimes N)$ to
$\rfunctilde_1 (M) \otimes \rfunctilde_1 (N)$.
\end{proof}

\subsection{Eigenspace splitting}
\label{subsect:eigenspace_rfunctilde1}

By Lemma \ref{lem:general_eigenspace},  there is a
splitting
$
 K_1 \cong K_1 ^+ \oplus K_1 ^-
$
 in $K_1^+ \dash
\unst$,  where $K_1^+ = H^* (BV_1)^{GL_1}$;  this induces a
splitting
 $
 \Philocal_1^{\sltilde_1} \cong \boldGamma_1 \oplus
(\Philocal_1^{\sltilde_1})^-
 $
in $\boldGamma_1\dash \amod$.

\begin{proposition}
\label{prop:rfunctilde_splitting}
For $M \in \obj \amod$, there is a natural splitting in  $K_1 ^+ \dash \amod$:
\[
 \rfunctilde_1 M \cong \rfunc_1 M \oplus \rfunc_1^- M , 
\]
where $\rfunc_1 M := \rfunctilde_1 M \cap \Big( \Philocal_1 ^{GL_1} \largetensor
M \Big)$ and  $\rfunc_1^- M \cong
\rfunctilde_1 M
\cap \Big((\Philocal_1 ^{\sltilde_1})^- \largetensor M\Big)$. In particular
$\rfunc_1 M$ and $\rfunc_1^- M$ are sub
$\cala$-modules of $\rfunctilde_1 M$.

There are natural isomorphisms of $K_1 ^+$-modules:
\begin{eqnarray*}
 \rfunc_1 M &\cong & K_1^+ \st_1 (M\even) \oplus K_1 ^- \st_1 (M\odd) \\
\rfunc_1^-  M &\cong & K_1^- \st_1 (M\even) \oplus K_1 ^+ \st_1 (M\odd).
\end{eqnarray*}
\end{proposition}

\begin{proof}
Straightforward.
\end{proof}

\begin{remark}
For  $M \in \obj \unst$, $\rfunc_1 M $ coincides with Zarati's  $R_1 M
$ \cite{zarati_these}. 
\end{remark}

\begin{corollary}
\label{cor:rfunc_1_weak_cts}
 The associations $M \mapsto \rfunc_1 M$ and $M \mapsto \rfunc_1^-M$ define
functors
\begin{eqnarray*}
 \rfunc_1 , \rfunc_1^- : \amod & \rightrightarrows & K_1 ^+ \dash \amod
\end{eqnarray*}
which are exact, commute with colimits, are weakly continuous and
$\kappa_1$-connective.
\end{corollary}

\begin{proof}
 The result follows from Proposition \ref{prop:rfunctilde_splitting} and
Proposition \ref{prop:rfunctilde1_gen_properties}.
\end{proof}

\begin{proposition}
\label{prop:rfunc1_suspension}
 For $M \in \obj  \amod$, there is a natural isomorphism $
 \rfunctilde_1(\Sigma M) \cong \Sigma \euler_1 \rfunctilde_1 M.
$ in $K_1\dash \amod$. This restricts to natural isomorphisms of $K_1
^+$-modules:
\begin{eqnarray*}
 \rfunc_1 (\Sigma M) &\cong &\Sigma \euler_1 \rfunc_1^- M \\
\rfunc_1^-(\Sigma M) &\cong &\Sigma \euler_1 \rfunc_1 M.
\end{eqnarray*}
In particular, $\rfunc_1 (\Sigma M) $ is a sub $\cala$-module of $\Sigma
\rfunc_1 M$. 
\end{proposition}

\begin{proof}
 Straightforward.
\end{proof}

\subsection{Multiplicative properties}

The multiplicativity of $S_1$ (cf. Proposition \ref{prop:S1_properties})
implies the following, using the multiplicative structures of 
Proposition \ref{prop:bullet-product}.

\begin{proposition}
\label{prop:rfunc1_mult_properties}
For $B$  an algebra in $\amod$ and $M \in \obj B\dash \amod$:
\begin{enumerate}
 \item
$\rfunctilde_1 B$ has a natural $K_1$-algebra structure in $\amod$ and
$\rfunctilde_1 M$ has a natural
$\rfunctilde_1 B$-module structure in $\amod$;
\item
$\rfunc_1 B$ has a natural $K_1^+$-algebra structure in $\amod$ and $\rfunc_1 M$
has a natural
$\rfunc_1 B$-module structure in $\amod$.
\end{enumerate}
For $L$ an unstable algebra and $N \in \obj L\dash \unst$:
\begin{enumerate}
	\item
$\rfunctilde_1 L$ has a natural $K_1$-algebra structure in unstable algebras and
$\rfunctilde_1 N$ is naturally an
object of
$\rfunctilde_1 L \dash \unst$;
\item
$\rfunc_1 L$ has a natural $K_1^+$-algebra structure in unstable algebras and
$\rfunc_1 N$  is naturally an object of
$\rfunc_1 L\dash \unst$.
\end{enumerate}
\end{proposition}

\subsection{The transformation $\rho_1$ and the fundamental short exact
sequence}

Proposition \ref{prop:rfunc1_suspension} shows that, for $M \in \obj \amod$,
there is a natural monomorphism in $K_1^+\dash \amod$:
\[
 \Sigma ^{-1}\rfunc_1(\Sigma M) \hookrightarrow \rfunc_1 M   
\]
and the cokernel identifies, as a graded vector space, with $\st_1 (M \even)
\oplus
\oddgen_{1,0} \st_1(M\odd)$. This  can be analysed in $\amod$ by using the 
natural transformation
$\rho_1$ introduced below, which generalizes (up to sign) that defined by Zarati
\cite[Définition-Proposition 3.3.7]{zarati_these} for unstable modules.

\begin{definition}
\label{def:rho1}
 Let $\rho_1  : \rfunc_1 M \rightarrow \Sigma^{-2} \Phi \Sigma M $ be the linear
natural transformation which factors
over the  cokernel of $ \Sigma ^{-1}\rfunc_1(\Sigma M) \hookrightarrow \rfunc_1
M$, defined on generators
by
\begin{eqnarray*}
  \st_1 (m) &\mapsto& -\Sigma ^{-2} \Phi (\Sigma m) \\
\oddgen_{1,0} \st _1 (n) & \mapsto &  \Sigma ^{-2} \Phi (\Sigma n),
\end{eqnarray*}
where $m \in M\even$ and $n \in M \odd$.
\end{definition}

\begin{proposition}
\label{prop:rho1_ses}
 For $M \in \obj \amod$, there is a natural short exact sequence in $\amod$:
\begin{eqnarray}
\label{eqn:rho1_ses}
 0 \rightarrow
\Sigma ^{-1}\rfunc_1(\Sigma M) \rightarrow \rfunc_1 M
\stackrel{\rho_1}{\rightarrow}
\Sigma^{-2} \Phi \Sigma M
\rightarrow
0.
\end{eqnarray}
\end{proposition}

\begin{proof}
The surjectivity of $\rho_1$ is clear, and the exactness as graded vector spaces
follows. It remains to show that $\rho_1$ is $\cala$-linear; this is
a straightforward calculation, using the method of proof of Proposition
\ref{prop:rfunctilde_astable}. Namely, (using the notation of {\em loc. cit.})
the only terms  of $P^i \st_1 (x)$ which are non-zero correspond to the cases:
\begin{enumerate}
 \item
$\epsilon =0$, $|x|$ even and $i = pt$;
\item
$\epsilon=1$, $|x|$ odd and $i = pt +1$.  
\end{enumerate}
\end{proof}

\subsection{The Singer evaluation}

Writing $H^* (BV_1) \cong \Lambda (u) \otimes \field_p [v]$, 
$\Philocal_1$ is the algebra $\Lambda (w)
\otimes \field_p [v^{\pm 1}]$, where $w= \frac{u}{v}$ has degree $-1$. The
linear morphism
$ 
\partial :  \Philocal_1 \twoheadrightarrow \Sigma^{-1} \field_p
$
sending $w$ to the canonical generator is $\cala$-linear by a fundamental result
of Singer (cf. \cite[Proposition
2.2]{li_singer} for $p$ odd).

\begin{definition}
\label{def:d_M}
 For $M \in \obj \amod$, let
$
 d_M : \rfunc_1 M \rightarrow \Sigma^{-1} M
$
be the natural transformation defined by the composite:
\[
 \rfunc_1 M \hookrightarrow \Philocal_1 \largetensor M \stackrel{\partial
\largetensor M} {\rightarrow}
\Sigma^{-1} M .
\]
\end{definition}

Recall from Section \ref{subsect:destab_functor} the explicit description of
$\destab M$ for $M \in \obj \amod$ as the
quotient $\destab M \cong M / BM$.
The following result establishes the interest of the functor $\rfunctilde_1$,
corresponding to the origin of the instability condition.

\begin{proposition}
\label{prop:coker_dM_destab}
For $M \in \obj  \amod$, the cokernel of
$
 d _{ M} :  \rfunc_1 M \rightarrow  \Sigma^{-1} M
$ 
 is $\Sigma^{-1} \destab M $.
\end{proposition}

\begin{proof}
The element $\st_1 (x)$ can be written as
$
\sum
\pm w^{\epsilon} \euler_1^{|x| -2i} \otimes \beta ^\epsilon P^i (x).
$

There is a basis of $\rfunc_1 M$  consisting of elements of the form
\begin{enumerate}
 \item
$|x|$ even: $Q_{1,0}^n \st_1 (x)$ or $w Q_{1,0}^{n+1}\st_1 (x)$
\item
$|x|$ odd: $\euler_1^{2n+1}\st_1(x)$ or $w\euler_1^{2n+1}\st_1 (x)$, 
\end{enumerate}
where $n \geq 0$, and $x$ runs over  a homogeneous basis of $M$.

Calculation gives:
\begin{enumerate}
 \item
$|x|$ even, $d (Q_{1,0}^n \st_1 (x)) = (-1)^{n+1} \Sigma^{-1}   \beta
P^{\frac{|x|+2n}{2}} (x)$;
\item
$|x|$ even, $d (w Q_{1,0}^{n+1}\st_1 (x) )= (-1)^{n+1} \Sigma^{-1} P^{\frac{|x|+
2(n+1)}{2}}(x)$;
\item
$|x|$ odd, $d(\euler_1^{2n+1}\st_1(x)) = (-1)^n  \Sigma^{-1}  \beta
P^{\frac{|x|+ 2n+1}{2}} (x)$;
\item
$|x|$ odd, 
$d (w \euler_1^{2n+1}\st_1 (x)) = (-1)^{n+1} \Sigma ^{-1}P^{\frac{|x|+ 2n+1}{2}}
(x)$.
\end{enumerate}
It follows that the image of $d$ is equal to $\Sigma^{-1} BM$, whence the
result. 
\end{proof}

\subsection{A connecting morphism}

The morphism $\Sigma d_M :\Sigma  \rfunc_1 M \rightarrow  M $ is the tail of the
chain complex  introduced in Section
\ref{sect:chain_complex} and the short exact sequence
(\ref{eqn:rho1_ses}) of Proposition \ref{prop:rho1_ses} fits into a short
exact sequence of chain complexes. An  understanding of the
connecting morphism induced in homology is required, which is given by the
following

\begin{proposition}
\label{prop:connecting_lambda}
 For $M \in \obj \amod$, the following diagram commutes
\[
 \xymatrix{
\rfunc_1 \Sigma M
\ar@{^(->}[r]
\ar[d]_{d_{\Sigma M}}
&
\Sigma \rfunc_1 M
\ar@{->>}[r]^{\Sigma \rho_1}
\ar[d]^{\Sigma d_M}
&
\Sigma^{-1} \Phi \Sigma M
\\
M
\ar@{=}[r]
\ar@{->>}[d]
&
M
\\
\Sigma ^{-1} \destab (\Sigma M),
}
\]
in which the three-term sequences are exact.

The associated connecting morphism
$
 \Sigma^{-1} \Phi \Sigma M \rightarrow \Sigma^{-1} \destab (\Sigma M) 
$
is  induced by $ \Sigma^{-1} \lambda_{\Sigma M}  : \Sigma^{-1}  \Phi (\Sigma M)
\rightarrow M$.
\end{proposition}

\begin{proof}
Consider an element $\Sigma^{-1}\Phi \Sigma m$ of $\Sigma^{-1}\Phi \Sigma M$; if
$m$ is of even degree then, by the
definition of $\rho_1$, $\Sigma \st_1 (m) \in \Sigma \rfunc_1 M$ is a lift;
under the differential $\Sigma d_M$, this
maps to $ \beta P^{|m|/2} m$, as in the proof of Proposition
\ref{prop:coker_dM_destab}.

Similarly, if $m$ is of odd degree, $\Sigma w \euler_1 \st_1  (m)$ is a lift;
under the differential $\Sigma d_M$,
this maps to $ P^{(|m|+1)/2} m$.

Now $\Sigma m \in \Sigma M$  is of degree $|m|+1$. On passing
to the cokernel, the morphism is well-defined and, by inspection, is induced by
the morphism $\Sigma^{-1}
\lambda_{\Sigma M} $, as required.
\end{proof}

\section{The functors  $\rfunc_s$}
\label{sect:rfunctors}

\subsection{Higher total Steenrod powers}

As in \cite[Definition 2.4]{hung_sum}, for $M \in \obj \amod$ and $0<s \in
\nat$, 
the linear morphism $S_1 : M \rightarrow \Phi_1 \largetensor M$ can be iterated,
to define 
$
	S_s : M \rightarrow \Philocal_s \largetensor M.
$

\begin{remark}
This recursive interpretation of $S_1 \circ S_{s-1}$ as above, rather than as a
map to $\Philocal_1 \largetensor (\Philocal_{s-1} \largetensor M)$ is one of
the unavoidable technicalities of the subject. A verification is given in
\cite[Lemma 2.5]{hung_sum}. 
\end{remark}

The definition of $\st_1$ in Definition \ref{def:S1} extends to the following,
with respect to the chosen basis of $V_s$.

\begin{definition}
	For $M \in \obj \amod$ and $x \in M$, define 
$
	\st_s (x) := (-1) ^{s [\frac{|x|}{2}] } \euler_s ^{|x|} S_s (x),
$ 
so that $|\st_s (x)| = p^s x$.
\end{definition}

\begin{remark}
The morphism $\st_s$ can also be defined directly by iterating $\st_1$.
\end{remark}

\subsection{Iterating the functor $\rfunctilde_1$}

\begin{lemma}
\label{lem:iterate_rfunctilde1}
	For  $s \in \nat$,
\begin{enumerate}
	\item
$\rfunctilde_1^{\circ s} : \amod \rightarrow \amod$ is
exact, weakly continuous and
$\kappa_s$-connective, where $\kappa_s (n) : = p^s n$;
\item
$\rfunctilde_1^{\circ s} : \amod \rightarrow \amod$ takes
values in $\rfunctilde_1^{\circ s}\field_p \dash \amod$.
\end{enumerate}
\end{lemma}

\begin{proof}
	The first statement follows from Proposition
\ref{prop:composition_weakly_cts_connective} and Corollary
\ref{cor:rfunc_1_weak_cts};  the second is a consequence of
the multiplicative properties of $\rfunctilde_1$ given in Proposition
\ref{prop:rfunc1_mult_properties}.
\end{proof}

Fix a choice of basis for $V_s$,  giving an isomorphism $H^* (BV_1) \otimes
\ldots \otimes H^* (BV_1) \cong H^* (BV_s)$ and
recursively a monomorphism of unstable algebras $
 \rfunctilde_1^{\circ s} \field_p \hookrightarrow H^* (BV_s).
$ The following result is required for defining the functors $\rfunc_s$, since
it
allows restriction to invariants. 

\begin{proposition}
\label{prop:rtilde_iterate_Philocal_s}
For $M \in \obj \amod$ and  $s \in \nat$,
\begin{enumerate}
\item
	there is a natural monomorphism
$		\rfunctilde_1^{\circ s} M
		\hookrightarrow
		\Philocal_s \largetensor M
$
of functors $\amod \rightarrow \rfunctilde_1^{\circ s}\field_p \dash \amod$,
where the right hand side is an
$\rfunctilde_1^{\circ s} \field_p$-module via restriction
of the  $\Philocal_s$-module structure along  $\rfunctilde_1
^{\circ s} \field_p \hookrightarrow H^*
(BV_s) \hookrightarrow \Philocal_s$;
\item
the underlying  $\rfunctilde_1^{\circ s}\field_p$-module of
$\rfunctilde_1^{\circ s} M$, considered as a submodule of $\Philocal_s
\largetensor M$, is free on $\st_s M$. 
\end{enumerate}
\end{proposition}

\begin{proof}
The first statement is proved by induction upon $s$, the cases $s=0,1$ being
clear. Suppose that the natural monomorphism $\rfunctilde_1^{\circ s-1} M
\hookrightarrow \Philocal_{s-1} \largetensor M$ is defined; to exhibit the
analogous natural transformation for $s$, since the functors are weakly
continuous, it suffices to
restrict to modules  $M$ which are bounded above, by Proposition
\ref{prop:properties_weak_cts}. 
 In this case, Proposition \ref{prop:otimes_embedding} provides a monomorphism
of $\Philocal_s$-modules $\Philocal_s \otimes M
\hookrightarrow \Philocal_1 \largetensor (\Philocal_{s-1} \largetensor M)$ and
there is a factorization 
\[
\xymatrix{
\rfunctilde_1 (\rfunctilde_1^{\circ s-1} M) 
\ar@{^(->}[rr]
\ar@{^(->}[rd]
&&
\Philocal_1 \largetensor (\Philocal_{s-1} \largetensor M),
\\
&
\Philocal_s \otimes M 
\ar@{^(->}[ur]
}
\]
 where the image of $\rfunctilde_1 (\rfunctilde_1^{\circ s-1} M) $ in
$\Philocal_s \otimes M$  identifies with the free $\rfunctilde_1^{\circ s}
\field_p$-module on $\st_s M$;   by induction this is a straightforward
consequence of the identification of $\st_s$ as $\st_1 \circ \st_{s-1}$ and the
multiplicative property of $\st_1$. The second statement follows by passage to
the limit.
\end{proof}

\begin{definition}
\label{def:rfunctildes}
	For $s \in \nat$  and  $M \in \obj \amod$, denote by
\begin{enumerate}
	\item
$K_s$ the unstable algebra
	$
 K_{s} : = \rfunctilde_1^{\circ s} \field_p  \cap H^* (BV_{s}) ^{\sltilde_{s}};
$
\item
$\rfunctilde_s M $  the sub $K_s$-module
$\rfunctilde_1^{\circ s} M \  \cap \  \Philocal_s^{\sltilde_s} \largetensor M $.
\end{enumerate}
By convention, the functor $\rfunctilde_0$ is  the identity.
\end{definition}

\begin{remark}
The subobject $\rfunctilde_s M$ of $\Philocal_s \largetensor
M$ is independent of the choice of basis used in defining the embedding
$\rfunctilde_1^{\circ s} M \hookrightarrow \Philocal_s \largetensor M$.
\end{remark}

 The introduction of the  functors $\rfunctilde_s$ is not strictly necessary
for the constructions of the paper; however, they exhibit better formal
properties than the functors $\rfunc_s$ introduced below, which are worth
reviewing.

\begin{corollary}
\label{cor:rfunctilde_s_general}
For $s\in \nat$:
\begin{enumerate}
	\item
	$M \mapsto \rfunctilde_s M$ defines a functor
$\rfunctilde_s : \amod \rightarrow K_s\dash
\amod$, equipped with a natural transformation $\rfunctilde_s M \hookrightarrow
\Philocal_s \largetensor M $ of functors
with values in $K_s \dash \amod$;
\item
	the underlying $K_s$-module of $\rfunctilde_s M $ is free on $\st_s M$;
	\item
	the functor $\rfunctilde_s : \amod \rightarrow \amod$ is exact, commutes
with colimits, is weakly continuous and
is $\kappa_s$-connective, where $\kappa_s (n)= p^s n$;
\item
for $M, N \in \obj \amod$, there are natural
isomorphisms in $K_s \dash \amod$
\begin{eqnarray*}
 \rfunctilde_s (M \otimes N)
&\cong&
\rfunctilde_s (M) \otimes _{K_s} \rfunctilde_s (N)
\\
 \rfunctilde_s(\Sigma M)
& \cong &\Sigma \euler_s \rfunctilde_s M.
\end{eqnarray*}
\end{enumerate}
\end{corollary}

\begin{proof}
	An immediate consequence of Lemma \ref{lem:iterate_rfunctilde1}
and Proposition \ref{prop:rtilde_iterate_Philocal_s}.
\end{proof}

\subsection{The functors $\rfunc_s$}

As for $s=1$, the functor $\rfunc_s$ is constructed by an eigenspace
splitting (cf. Lemma
\ref{lem:general_eigenspace}). The action of $\zed/2$ on $H^*
(BV_{s}) ^{\sltilde_{s}}$ induces an action of $\zed/2$ on $K_s$ by morphisms of
unstable algebras, hence an eigenspace
decomposition
$
	K_s \cong K_s ^+ \oplus K_s^-
$
in  $K_s^+ \dash \unst$, which passes to $\Philocal_s ^{\sltilde_s} \cong
\boldGamma_s \oplus (\Philocal^{\sltilde_s} ) ^-$ in $\boldGamma_s \dash \amod$.

\begin{definition}
For $s \in \nat$ and $M \in \obj \amod$, let
\begin{enumerate}
\item
$\rfunc_s M$ be the sub $K_s^+$-module  $\rfunctilde_1^{\circ s}  M \cap \
\boldGamma_s \largetensor M = \rfunctilde_s M \cap \
\boldGamma_s \largetensor M$;
\item
$\rfunc_s ^- M$ be the sub $K_s^+$-module $\rfunctilde_1^{\circ s} M \cap \
(\Philocal_s^{\sltilde_s})^- \largetensor M = \rfunctilde_s M \cap \
(\Philocal_s^{\sltilde_s})^- \largetensor M$.
\end{enumerate}
\end{definition}

\begin{remark}
	The functor $\rfunc_0$ is the identity and $\rfunc_0^-=0$.
\end{remark}

\begin{proposition}
\label{prop:rfunc_s_properties}
For $s \in \nat$,
\begin{enumerate}
 \item
$M \mapsto \rfunc_s M$, $M \mapsto \rfunc_s^- M$ define
functors $\amod \rightarrow K_s^+ \dash \amod$;
\item
there is a natural monomorphism 
$
 \rfunc_s M \hookrightarrow \boldGamma_s \largetensor M
$ in $K_s^+
\dash \amod$;
\item
there is a natural decomposition $
 \rfunctilde_s M \cong \rfunc_s M \oplus \rfunc_s^- M
$ 
in $K_s^+ \dash \amod$ and,  as modules over $K_s^+$,
\begin{eqnarray*}
 \rfunc_s M & \cong & K_s^+ \st_ s(M\even) \oplus K_s^- \st_s (M \odd) \\
\rfunc_s^- M &\cong & K_s^- \st_s (M\even) \oplus K_s ^+\st_s (M\odd);
\end{eqnarray*}
\item
the functors $\rfunc_s$, $\rfunc_s^-$ are exact, commute with
colimits, are weakly continuous and $\kappa_s$-connective.
\end{enumerate}
\end{proposition}

\begin{proof}
 The first statement is clear; the second is a consequence of the restriction to
$GL_s$-invariants, which implies that
the constructions are independent
of choices.  The third identification follows from the fact that $\st_s (x)$ is
$\sltilde_s$-
invariant and is $GL_s$ invariant if and only if $|x|$ is even; the final
statement is an immediate consequence of Corollary
\ref{cor:rfunctilde_s_general}.
\end{proof}

\begin{remark}
	This recovers, in the case $M$ unstable, the definition of the functor
$R_s$ given by Zarati \cite[Définition
2.4.5]{zarati_these}, subject to the identification of the unstable algebra
$K_s$ given in Section \ref{sect:ks} below.
\end{remark}

The behaviour of $\rfunc_s$ with respect to tensor products  is
slightly more complicated than for $\rfunctilde_s$.

\begin{proposition}
\label{prop:rfuncs_suspension}
 For $M \in \obj \amod$, there are natural isomorphisms of $K_s ^+$-modules:
\begin{eqnarray*}
 \rfunc_s (\Sigma M) &\cong &\Sigma \euler_s \rfunc_s^- M \\
\rfunc_s^-(\Sigma M) &\cong &\Sigma \euler_s \rfunc_s M.
\end{eqnarray*}
In particular, $\rfunc_s (\Sigma M) $ is a sub $\cala$-module of $\Sigma
\rfunc_s M$.
\end{proposition}

By construction, for $M \in \obj \amod$, there are natural inclusions $\rfunc_s
M \hookrightarrow \rfunctilde_s M \hookrightarrow \Philocal_s \largetensor M$,
and an induced natural inclusion $\rfunc_1^{\circ s} M\hookrightarrow
\rfunctilde_1 ^{\circ s}M$, for $s \in \nat$.

\begin{proposition}
\label{prop:canonical_embeddings}
Let $s$ be a natural number.
\begin{enumerate}
\item
There is a natural embedding $\rfunc_s \hookrightarrow \rfunc_1^{\circ s}$ which
fits into a commutative diagram of natural transformations of functors with
values in $K^+_s \dash \amod$:
\[
 \xymatrix{
\rfunc_s 
\ar@{^(->}[r]
\ar@{^(->}[d]
&
\rfunc_1^{\circ s} 
\ar@{^(->}[d]
\\
\rfunctilde_s 
\ar@{^(->}[r]
&
\rfunctilde_1^{\circ s} 
\ar@{^(->}[r]
&
\Philocal_s \largetensor - .
}
\]
\item
For $M \in \obj \amod$,  $
\rfunc_{s+2} M
=
\bigcap _{i+j = s}
\rfunc_1 ^{\circ i} \rfunc_2 \rfunc_1^{\circ j} M.
$ as submodules of $\rfunc_1 ^{\circ s} M$.
\end{enumerate}
\end{proposition}

\begin{proof}
The first statement is a straightforward verification. The second follows from
the corresponding
result for $GL_s$-invariants:
\[
 H^* (BV_s)^{GL_s}  \cong
\bigcap _{i+j = s-2}
H^* (BV_s)^{GL_2^{i,j}} 
\]
where  $GL_2 \cong GL_2^{i,j}
\subset GL_s$ acts on the factor $\field_p^{\oplus 2}$ in  the  direct
sum decomposition induced by the basis
$
V_s \cong 
\field_p^{\oplus i}
\oplus
\field_p^{\oplus 2}
\oplus
\field_p^{\oplus j}.
$
\end{proof}

\begin{remark}
 The second statement of the Proposition exhibits the quadratic nature of the
functors $\rfunc_s$:  they are determined by the
functor $\rfunc_1$ and the inclusion $\rfunc_2 \hookrightarrow \rfunc_1
\rfunc_1$.
\end{remark}

As a consequence, one obtains:

\begin{corollary}
 \label{cor:rfunc_composites}
For $s,t \in \nat$ and $M \in \obj \amod$,
\begin{enumerate}
 \item
the inclusion $\rfunc_{s+t} \hookrightarrow \rfunc_1^{\circ s+t}$ factors as 
$
 \rfunc_{s+t} \hookrightarrow 
\rfunc_s \rfunc _t 
\hookrightarrow
\rfunc_1^{\circ s+t};
$
\item
the commutative diagram
\[
 \xymatrix{
\rfunc_{s+2} M
\ar@{^(->}[r]
\ar@{^(->}[d]
&
\rfunc_1 \rfunc_{s+1} M 
\ar@{^(->}[d]
\\
\rfunc_{s+1} \rfunc_{1} M
\ar@{^(->}[r]
&
\rfunc_1 \rfunc_{s} \rfunc_{1} M 
}
\]
is cartesian.
\end{enumerate}
\end{corollary}

\section{The unstable algebra $K_s$}
\label{sect:ks}

The explicit description of the functor $\rfunc_s$, for $s \in \nat$, is
completed
 by identifying the unstable algebra $K_s$ and, via the eigenspace
decomposition,
the unstable
 algebra $K_s^+$ and the $K_s^+$-module $K_s^-$. Of necessity, this involves
some calculation. A direct
 approach is taken by Zarati in \cite{zarati_these}; here an alternative method
is indicated, based on the results of M{\`u}i \cite{mui_cohom_operations}.

\subsection{The unstable algebras $K_s, K_s^+$ and the module $K_s^-$}

Recall the unstable algebra of Theorem \ref{thm:sltildes_inv}, which also
gives the identification used in the following statement.

\begin{proposition}
(Cf. \cite[Proposition 2.4.7.2]{zarati_these}.)
\label{prop:Ks_rfunctilde1}
	For $s \in \nat$,
\begin{eqnarray*}
 K_{s} &=& \mathrm{Image} \{ H^* (B \alt_{p^s}) \rightarrow H^*(BV_s) \} \\
&=& \Lambda (\oddgen_{s,i} | 0 \leq i \leq s-1 ) \otimes \field _p
[\euler_s, Q_{s,j} | 1 \leq j \leq s-1 ].
\end{eqnarray*}
\end{proposition}

\begin{proof} (Indications.)
For the purposes of this proof, define a functor $\rfunctilde'_1 : \amod
\rightarrow H^* (BV_1) \dash \amod$ by
extension of scalars
\[
 \rfunctilde'_1 M := H^* (BV_1) \otimes_{K_1} \rfunctilde_1 M .
\]
It is clear that $\rfunctilde'_1$ restricts to an endofunctor of $\unst$ which
has good multiplicative properties. In
particular, if $K$ is an unstable algebra, then
$\rfunctilde'_1 K$ is an unstable algebra and there is an inclusion of unstable
algebras
 $
 \rfunctilde_1 K \hookrightarrow \rfunctilde'_1 K.
$

In \cite[Theorem 3.9]{mui_cohom_operations}, M{\`u}i defines an algebra
$\mathscr{M}_p (s)$ for $s\in \nat$, which
is a free graded algebra
\[
 \mathscr{M}_p (s)
\cong
\Lambda (U_1, \ldots , U_s) \otimes \field_p [V_1, \ldots , V_s]
\]
on explicit generators defined in \cite[Section 2]{mui_cohom_operations}.
Combining \cite[Proposition 2.6]{mui_cohom_operations}
with  \cite[Theorem 3.8]{mui_cohom_operations},
one can show that $\mathscr{M}_p (s) $ is isomorphic to $(\rfunctilde'_1)
^{\circ s} \field_p$ (hence is an unstable
subalgebra of $H^* (BV_s)$). In particular, there
is an inclusion of unstable algebras
 $
 \rfunctilde_1^{\circ s} \field_p \hookrightarrow \mathscr{M}_p (s).
$
By \cite[Lemma 3.11]{mui_cohom_operations}
\[
 \mathscr{M}_p (s) \cap H^* (BV_s)^{\sltilde_s}
\cong
\Lambda (\oddgen_{s,i} | 0 \leq i \leq s-1 ) \otimes \field _p [\euler_s,
Q_{s,j} | 1 \leq j \leq s-1 ].
\]

Since  $\mathscr{M}_p (s) = (\rfunctilde'_1)
^{\circ s} \field_p$,  it follows that $\mathscr{M}_p (s) \cap H^*
(BV_s)^{\sltilde_s} = \rfunctilde_1 ^{\circ s} \field_p$, which is $K_s$, by
definition.
\end{proof}

\begin{nota}
\label{nota:Mtilde_I_odd_even}
For $0< s \in \nat$ and $I \subset \{ 0, \ldots , s-1\}$,
\begin{enumerate}
 \item
let $\oddgen_{s,I}$ denote the monomial $\prod _{i \in I} \oddgen_{s,i}$, where
the factors in the product are ordered by the natural order on $I$ (in particular $\oddgen_{s,\emptyset}=1$);
\item
write $I\even$ (respectively $I\odd$) if $|I|$ is even (resp. odd) and
$\oddgen_{s,I\even}$ (resp. $\oddgen_{s,I\odd}$) for the associated monomials.
\end{enumerate}
\end{nota}

\begin{corollary}
\label{cor:Kplus_Kminus}
For $s \in \nat$, there is an inclusion of unstable algebras $\field_p [Q_{s,j}
| 0 \leq j \leq s-1]\hookrightarrow K_s^+$ and 
\begin{enumerate}
 \item
$K_s^+$ is the free $\field_p [Q_{s,j} ]$-module on
 $\{
\euler_s \oddgen_{s,  I \odd }, \oddgen_{s, I\even} \}$;
\item
$K_s^-$ is the free $\field_p [Q_{s,j} ]$-module on $\{
\euler_s \oddgen_{s,  I \even }, \oddgen_{s, I\odd} \}$.
\end{enumerate}
\end{corollary}

\begin{proof}
From their construction (cf. Definition \ref{def:invariants} and Proposition \ref{prop:GL_sltilde_invariance}) it is clear that $\euler_s$ and the $\oddgen_{s,i}$ belong to $K_s^-$ and that the $Q_{s,j}$ belong to $K_s^+$, hence  the stated conclusion follows from  Proposition \ref{prop:Ks_rfunctilde1}.
\end{proof}

\section{The chain complex}
\label{sect:chain_complex}
\subsection{The complex}

Recall from Definition \ref{def:d_M} the natural transformation $d_M : \rfunc_1
 M \rightarrow \Sigma^{-1} M$.

\begin{definition}
\label{def:diff_ds}
 Let $d_s : \rfunc_s \Sigma M \rightarrow  \rfunc_{s-1}M $ be the composite in
$\amod$:
\[
 \rfunc_s \Sigma M \hookrightarrow \rfunc_{s-1} (\rfunc_{1} \Sigma M)
\stackrel{\rfunc_{s-1} d_{\Sigma M}}{\longrightarrow}
\rfunc_{s-1}M.
\]
\end{definition}

\begin{definition}
 For $M \in \obj \amod$, let $\dcx_\bullet M$ denote the chain complex in
$\amod$ with 
$
 \dcx _s M := \Sigma \rfunc_s \Sigma^ {s-1}M ,
$ and  differential $\dcx_s M \rightarrow \dcx_{s-1} M$ induced by $d_s :
\rfunc_s \Sigma \rightarrow \rfunc_{s-1}$.
\end{definition}

\begin{remark}
That this is indeed a chain complex (namely $d^2= 0$) is established in
Proposition \ref{prop:chcx}. 
\end{remark}

The chain complex $\dcx_\bullet M$ has the form
\[
 \ldots
\rightarrow 
\Sigma \rfunc_3 (\Sigma^2 M)
\rightarrow
\Sigma \rfunc_2 (\Sigma M)
\rightarrow
\Sigma \rfunc_1 M
\rightarrow
M \rightarrow 0.
\]

\begin{proposition}
\label{prop:dcx_exact}
 The functor $\dcx_\bullet : \amod \rightarrow \chcx_{\geq 0} \amod
$ is
exact. 
\end{proposition}

\begin{proof}
Follows from  the exactness of $\rfunc_s$, by Proposition
\ref{prop:rfunc_s_properties}.
\end{proof}

\subsection{The vanishing of $d^2$}
\label{subsect:d2_vanish} 
The fact that $d_s$ defines a chain complex
is a consequence of the relationship between the Steenrod algebra and the
theory of invariants established by Singer \cite{singer_invt_lambda} for $p=2$
and developed by Hung and Sum \cite{hung_sum} (amongst others) for $p$ odd.
This can be proved by direct calculation; here an approach is taken which 
exhibits  the relationship with the chain complex considered by
\cite{hung_sum}.

\begin{lemma}
\label{lem:T2_invariance}
 The natural transformation $\rfunctilde_1 \rfunctilde_1 M \hookrightarrow
\Philocal_2 \largetensor M$ factors across the inclusion $\boldDelta_2
\largetensor M \hookrightarrow \Philocal_2 \largetensor M$. 
\end{lemma}

\begin{proof} Since $\st_2$ takes values in
$\Philocal_2^{\sltilde_2} \largetensor M$, by Proposition
\ref{prop:rtilde_iterate_Philocal_s} it suffices to show that
$\rfunctilde_1 \rfunctilde_1 \field \subset H^* (BV_2)^{T_2}$. By the exactness
of $\rfunctilde_1$, it suffices to show that $\rfunctilde_1 H^* (BV_1)$ lies in
$H^* (BV_2)^{T_2}$.

Using the multiplicativity of $\st_1$, one reduces further to showing that the
elements $\st_1 (x)$ and $\st_1 (y)$  are $T_2$-invariant, where $x, y$ are the
algebra generators of $H^* (BV_1)$; this can be verified by direct calculation.
\end{proof}

The Singer evaluation $\boldDelta_1= \Philocal_1 \rightarrow \Sigma^{-1} \field$
induces a
morphism 
$ 
 \boldDelta_2 \rightarrow \boldDelta_1 \otimes \Sigma^{-1} \field
$ 
 \cite[Definition 3.5 and Proposition 3.6(i)]{hung_sum}. For $M \in \obj \amod$,
by forming $-\largetensor M$,  this defines
$
 \boldDelta_2 \largetensor M \stackrel{\partial_{\boldDelta}}{\longrightarrow}
\boldDelta_1 \largetensor \Sigma^{-1} M
$
\cite[Definition 4.1(i)]{hung_sum}.

The following result relates  $\partial_\boldDelta$ to the
differential of the chain complex.

\begin{lemma}
\label{lem:compatibility_diff}
 For $M \in \obj \amod$, the following diagram commutes:
\[
 \xymatrix{
\rfunctilde_1 \rfunctilde_1 M 
\ar[r]^{\rfunctilde_1 d_M} 
\ar@{^(->}[d]
&
\rfunctilde_1 (\Sigma^{-1} M) 
\ar@{^(->}[d]
\\
\boldDelta_2 \largetensor M 
\ar[r]_{\partial _\boldDelta} 
&
\boldDelta_1 \largetensor \Sigma^{-1} M.
}
\]
\end{lemma}

\begin{proof}
It is straightforward to check that the morphisms are $K_1 = \rfunctilde_1
\field$-linear, with respect to the evident module structures. Hence it
suffices to consider the module generators $\st_1 (z) \in \rfunctilde_1
\rfunctilde_1 M$, for $z \in \rfunctilde_1 M$. A further simplification is
obtained by localizing, inverting $\euler_1 \in K_1$, so that $\st_1(z)$ can be
replaced by
$S_1 (z)$.

Using the weak continuity of the functors, we may assume that $M$ is bounded
above, hence  $z$ is an element of $\Philocal_1 \otimes
M$. Consider a general  element $\phi \otimes m \in \Philocal_1 \otimes M$; a
straightforward generalization of Lemma \ref{lem:T2_invariance} shows that $S_1
(\phi \otimes m)$ lies in $\boldDelta_2 \largetensor M$, hence it suffices to
prove the more general compatibility replacing $z$ by $\phi \otimes m$.  

Passing around the top of the diagram sends $S_1 (\phi \otimes m)$ to $S_1
((\partial\phi)m)$, interpreted in the obvious way. To calculate the passage
around the bottom of the diagram, one uses the multiplicativity of $S_1$. 
That one obtains the same element in $\boldDelta_1 \largetensor
\Sigma^{-1} \field $ follows  from the commutativity of the diagram
\[
	\xymatrix{
\Philocal_1
\ar[r]^{S_1}
\ar[d]_\partial
&
\Philocal_1 \largetensor \Philocal_1
\ar[d]^{1 \largetensor \partial}
\\
\Sigma^{-1} \field_p
\ar[r]^(.4)\eta
&
\Philocal_1 \otimes\Sigma^{-1} \field_p,
}
\]
where $\eta : \Sigma^{-1} \field_p \rightarrow \Philocal_1 \otimes
\Sigma^{-1} \field_p$ is induced by the unit of $\Philocal_1$. This is
simply a restatement of the fact that $\partial : \Philocal_1 \rightarrow
\Sigma^{-1} \field$ is $\cala$-linear.
\end{proof}

\begin{lemma}
\label{lem:triviality_dd}
 For $M \in \obj  \amod$, the composite $
 \rfunc_2 \Sigma ^2 M 
\stackrel{d_2}{\rightarrow}
\rfunc_1 \Sigma M
\stackrel{d_1}{\rightarrow}
M 
$ is trivial.
\end{lemma}

\begin{proof}
 The commutative diagram of Lemma \ref{lem:compatibility_diff} restricts using
the inclusion $\rfunc_1 \hookrightarrow \rfunctilde_1$; the result  fits
into the commutative diagram:
\[
 \xymatrix{
\rfunc_2 \Sigma^2 M
\ar[d]
\ar[r]
&
\rfunc_1 \rfunc_1 \Sigma^2 M 
\ar[d]
\ar[r]
&
\rfunc_1 \Sigma M
\ar[d]
\ar[r]
&
M
\ar@{=}[d]
\\
\boldGamma_2 \largetensor \Sigma^2 M
\ar[r]
\ar@/_1pc/[rrd]
&
\boldDelta_2 \largetensor \Sigma^2 M 
\ar[r]
&
\boldDelta_1 \largetensor \Sigma M 
\ar[r]
&
M
\\
&&
\boldGamma_1 \largetensor \Sigma M 
\ar[u]
\ar@/_1pc/[ur]
}
\]
where the commutativity of the bottom part of the diagram follows from the
first statement of \cite[Proposition 3.6(ii)]{hung_sum}, with the curved arrows
induced by $\partial_2 : \boldGamma_2 \rightarrow \boldGamma _1 \otimes
\Sigma^{-1} \field$ and $\partial_1 : \boldGamma_1  \rightarrow \Sigma^{-1}
\field$, in the notation of {\em loc. cit.} (but making the desuspensions
explicit). Finally,  \cite[Proposition 3.6(ii)]{hung_sum} implies that the
composite of the
curved arrows is zero, which completes the proof.
\end{proof}

\begin{proposition}
\label{prop:chcx}
For $M \in \obj \amod$, $\dcx_\bullet M$ is a chain complex in $\amod$.
\end{proposition}

\begin{proof}
A straightforward (and standard) reduction, based on the quadratic nature of the
construction of $\dcx_\bullet M$,  shows
that it is sufficient to prove that the composite 
$
 \rfunc_2 \Sigma ^2 M 
\stackrel{d_2}{\rightarrow}
\rfunc_1 \Sigma M
\stackrel{d_1}{\rightarrow}
M 
$
is trivial; this is proved as Lemma  \ref{lem:triviality_dd} above.
\end{proof}

\begin{remark}
\label{rem:Gamma_plus}
The above analysis of the differential also serves to establish that, 
 for $M \in \obj \amod$, there is a natural monomorphism of chain complexes 
$
 \dcx_\bullet M \hookrightarrow \Gamma^+ _\bullet M,
$ 
where $\Gamma^+_\bullet M$ is the chain complex of \cite[Section 4]{hung_sum}.
\end{remark}

\section{The natural transformations $\rho_s$ and fundamental short exact
sequences}
\label{sect:ses}

\subsection{The higher natural transformation $\rho_s$}
\label{subsect:rhos}

Recall from Definition \ref{def:rho1} the morphism $\rho_1 : \rfunc_1 M
\rightarrow \Sigma^{-2} \Phi \Sigma M$. 

\begin{definition}
 For $M \in \obj \amod$ and $0< s \in \nat$, let $\rho_s :
\rfunc_s M \rightarrow \Sigma^{-2} \Phi \Sigma
\rfunc_{s-1}M$ be the natural
transformation given by the composite:
\[
 \rfunc_s M
\hookrightarrow
\rfunc_1 \rfunc_{s-1} M
\stackrel{\rho_1} {\rightarrow}
 \Sigma^{-2} \Phi \Sigma \rfunc_{s-1}M.
\]
\end{definition}

The morphism $\rho_s$ features in the short exact sequence of complexes which
is the key to the proof of the main result of the paper.

\subsection{Linearity of $\rho_s$ over $\field_p [Q_{s,i}]$}

The following is straightforward:

\begin{lemma}
\label{lem:Phi-module-structure}
 Let  $K$ be an unstable algebra concentrated in even degrees. For  $M \in K
\dash \amod$, 
$\Sigma^{-2} \Phi \Sigma M$ has a  natural  $\Phi K$-module structure defined by
$
	(\Phi k )( \Sigma^{-2} \Phi \Sigma m) = \Sigma ^{-2} \Phi \Sigma (km).
$
\end{lemma}

\begin{lemma}
\label{lem:dickson_phi_s}
\cite[Proposition 2.9]{kameko_mimura}
For $0< s\in \nat$  and a linear monomorphism  $i : V_{s-1}\hookrightarrow V_s
$,
the following commutes in unstable algebras:
\[
\xymatrix{
\field_p [Q_{s,0}, \ldots, Q_{s, s-1}]
\ar[r]^{\phi_s}
\ar@{^(->}[d]
&
\field_p [Q_{s-1,0}, \ldots, Q_{s-1, s-2}]
\ar@{^(->}[d]
\\
H^* (BV_s)
\ar@{->>}[r]_{i^*}
&
H^* (BV_{s-1})
}
\]
 where
$
 \phi_s Q_{s,0} =0$ and  $\phi_s Q_{s,j} =
Q_{s-1, j-1} ^p$, for  $ j >0$. 
In particular, $\phi_s$ factorizes as:
\[
 \field_p [Q_{s,0}, \ldots, Q_{s, s-1}]
\stackrel{\overline{\phi}_s}{\twoheadrightarrow}
\Phi
\field_p [Q_{s-1,0}, \ldots, Q_{s-1, s-2}]
\hookrightarrow
\field_p [Q_{s-1,0}, \ldots, Q_{s-1, s-2}].
\]
\end{lemma}

Corollary \ref{cor:Kplus_Kminus} provides an inclusion of unstable algebras
$\field_p [Q_{s,i}] \hookrightarrow K_s^+$, so that $\rfunc_s M $ is an object
of $\field_p[Q_{s,i}]\dash \amod$;  $\Sigma^{-2} \Phi \Sigma \rfunc_{s-1} M$ is
also an object of $\field_p[Q_{s, i}]\dash\amod$ by restriction along
$\overline{\phi}_s$ of the $\Phi \field_p[Q_{s-1,j}]$-module structure provided
by Lemma \ref{lem:Phi-module-structure}.

The following result  can be compared with \cite[Sections 4.3.2 and
4.6]{zarati_these}.

\begin{proposition}
\label{prop:linearity_rho}
For $M \in \obj \amod$ and $0<s\in \nat$,
$
	\rho_{s} : \rfunc_{s} M \rightarrow \Sigma^{-2} \Phi \Sigma
\rfunc_{s-1} M
$
is a morphism of $\field_p [Q_{s,i}| 0  \leq i \leq s-1]$-modules.
\end{proposition}

\begin{proof}
By Proposition \ref{prop:rfunc1_mult_properties}, $\rfunc_1 \field_p
[Q_{s-1,j}]$
is an unstable algebra and there is an inclusion of unstable algebras
$\field_p [Q_{s,i}] \hookrightarrow \rfunc_1 \field_p[Q_{s-1,j}]$. It is a
standard calculation to show that
\begin{eqnarray*}
	Q_{s,0} & = & Q_{1,0} \st_1 (Q_{s-1}, 0) \\
	Q_{s,i} & = & Q_{1,0}^{p^i} \st_1 (Q_{s-1,i} ) + \st_1 (Q_{s-1, i-1}),
\end{eqnarray*}
where $0< i < s$.

To prove the result, by restriction along the $\field_p [Q_{s,i}]$-linear
inclusion $\rfunc_s M \hookrightarrow \rfunc_1 \rfunc_{s-1}M$,  it suffices to
show that $\rho_1 : \rfunc_1 N \rightarrow \Sigma^{-2} \Phi \Sigma N$ is
$\field_p [Q_{s,i}]$-linear for $N \in \obj \field_p[Q_{s-1,j}]\dash \amod$,
where $\rfunc_1 N$ is a module by restriction of the $\rfunc_1 
\field_p[Q_{s-1,j}]$-module structure of Proposition
\ref{prop:rfunc1_mult_properties} along $\field_p [Q_{s,i}] \hookrightarrow
\rfunc_1 \field_p[Q_{s-1,j}]$ and the module structure of $\Sigma^{-2} \Phi
\Sigma N$ is as above.

Using the explicit description of $\rho_1$ given in  Definition \ref{def:rho1},
this is a straightforward consequence of the multiplicativity of $\st_1$ and the
observation that, modulo $Q_{1,0}$,  $Q_{s,0} \equiv 0 $ and $Q_{s,i}\equiv
\st_1 (Q_{s-1, i-1})$ for $i>0$, by the above equations.
\end{proof}

\subsection{The fundamental short exact sequence}

\begin{proposition}
\label{prop:rho_s_fund_exact}
(Cf. \cite[Définition-Proposition 4.5.1]{zarati_these}.)
For $M \in \obj \amod$ and $0< s \in \nat$, there is a natural short
exact sequence in $\field_p [Q_{s,i}]\dash \amod$:
\[
 0
\rightarrow
\Sigma^{-1}
\rfunc_s (\Sigma M)
\rightarrow
\rfunc_s M
\stackrel{\rho_s}{\rightarrow}
\Sigma^{-2} \Phi \Sigma \rfunc_{s-1}M
\rightarrow
0.
\]
\end{proposition}

A key reduction is given by the following result:

\begin{lemma}
\label{lem:rho_s_comm_diag}
For $M \in \obj \amod$ and $0< s \in \nat$, there is a commutative diagram of
exact sequences
\[
	\xymatrix{
0
\ar[r]
&
\Sigma^{-1}\rfunc_s (\Sigma M)
\ar[r]
\ar@{^(->}[d]
&
\rfunc_s M
\ar[r]^(.4){\rho_s}
\ar@{^(->}[d]
&
\Sigma^{-2} \Phi \Sigma \rfunc_{s-1}M
\ar@{=}[d]
\\
0
\ar[r]
&
\Sigma^{-1}\rfunc_1 \Sigma (\rfunc_{s-1} M)
\ar[r]
&
\rfunc_1 \rfunc_{s-1} M
\ar[r]^(.45){\rho_1}
&
\Sigma^{-2} \Phi \Sigma \rfunc_{s-1}M
\ar[r]
&
0.
}
\]
\end{lemma}

\begin{proof}
	The bottom row is provided by Proposition \ref{prop:rho1_ses} and
$\rfunc_s M \hookrightarrow \rfunc_1 \rfunc_{s-1} M $ is the inclusion of
Corollary \ref{cor:rfunc_composites}. A standard calculation (compare
\cite[Définition-Proposition 4.5.1]{zarati}) shows that the left hand square is
a pullback. The exactness of the top row follows. 
\end{proof}

The proof of Proposition \ref{prop:rho_s_fund_exact} is  completed by showing
that $\rho_s$ is surjective. This relies upon the knowledge of the structure of
$K_s^+$ and $K_s^-$ (see Corollary \ref{cor:Kplus_Kminus}). Indeed, the
necessary calculational input is already contained within the
calculation of $K_s$ in Proposition \ref{prop:Ks_rfunctilde1}. This allows for
the following quick  proof.

\begin{proof}[of Proposition \ref{prop:rho_s_fund_exact}]
	The functors appearing in the putative short exact sequence are weakly
continuous, exact and commute with colimits. Hence, by Proposition
\ref{prop:properties_weak_cts}, one reduces to the case where $M$ is of the form
$\Sigma^t \field_p$, for $t \in \zed$. Moreover, to prove exactness, it suffices
to consider the underlying graded vector spaces; these are of finite type, hence
by the exactness of the top row of Lemma \ref{lem:rho_s_comm_diag}, it suffices
to check that the Euler-Poincaré characteristic of the sequence is zero. This is
a direct verification.
\end{proof}

\begin{remark}
An alternative proof is to show the surjectivity of $\rho_s$ directly, as in
\cite[Définition-Proposition 4.5.1]{zarati_these}, using the explicit
identification of $K_s^+$ and $K_s^-$ and exploiting the
$\field_p[Q_{s,i}]$-linearity established in Proposition
\ref{prop:linearity_rho}.

By an argument similar to that of Corollary \ref{cor:Kplus_Kminus}, 
the cokernel of $\Sigma^{-1} \rfunc_s \Sigma M \hookrightarrow \rfunc_s M $ is a free $\field_p [Q_{s,j} | 1\leq j \leq
s-1]$-module  on elements of the form  $\oddgen_{s, I\even } \st_s (m{\even})$ or  $\oddgen_{s, I\odd }  \st_s (m{\odd})$
as $m$ runs through a basis of $M$, with the superscript representing the parity
of $|m|$. 

The elements $\oddgen_{s,I}$  can be written  in one of the following
forms: 
\begin{eqnarray*}
\oddgen_{s,0} \oddgen_{s, J\odd} \mathrm{\ or \ } \oddgen_{s,J\even}, \  |I|\equiv 0 \  (2)\\
\oddgen_{s,0} \oddgen_{s,J\even} \mathrm{\  or \ } \oddgen_{s,J\odd},\  |I|\equiv 1 \  (2)\\
\end{eqnarray*}
where  $J\subset \{1, \ldots , s-1 \}$.

As in the proof of  Proposition \ref{prop:linearity_rho},  the multiplicativity
of $\st_1$ together with the identification of the inclusion $K_s
\hookrightarrow \rfunctilde_1 K_{s-1}$ give:
\begin{eqnarray*}
 \oddgen_{s,0} \oddgen_{s,J\odd} \st_s (m{\even})&\mapsto & \pm \Sigma^{-2} \Phi
\Sigma \euler_{s-1}
\oddgen_{s-1,\tilde{J}\odd}\st_{s-1} (m{\even})
\\
\oddgen_{s,J\even}\st_s (m{\even}) &\mapsto & \pm \Sigma^{-2} \Phi \Sigma
\oddgen_{s-1,\tilde{J}\even}
\st_{s-1} (m{\even})
\\
 \oddgen_{s,0} \oddgen_{s,J\even} \st_s (m{\odd})&\mapsto & \pm \Sigma^{-2} \Phi
\Sigma^{}
\euler_{s-1}\oddgen_{s-1,\tilde{J}\even} \st_{s-1} (m{\odd})
\\
\oddgen_{s,J\odd} \st_s (m{\odd})&\mapsto & \pm \Sigma^{-2} \Phi \Sigma
\oddgen_{s-1,\tilde{J}\odd}\st_{s-1} (m{\odd}),
\end{eqnarray*}
where, for $J\subset \{1, \ldots , s-1 \}$,  $\tilde{J}$ is the subset of $\{ 0,
\ldots , (s-1)-1 \}$ which is induced by the order preserving surjection $ \{1, \ldots , s-1 \}\twoheadrightarrow \{
0, \ldots , (s-1)-1 \}$.

The surjectivity of $\rho_s$ follows.
\end{remark}

\subsection{The short exact sequence of chain complexes}

\begin{theorem}
\label{thm:ses_complexes}
For $M \in \obj \amod$, there is a natural short exact sequence of chain
complexes in $\amod$:
\[
 0
\rightarrow
\Sigma ^{-1}\dcx_\bullet \Sigma M
\rightarrow
\dcx _\bullet M
\stackrel{\rho_\bullet}{\rightarrow}
\Sigma^{-1} \Phi \dcx_{\bullet - 1}\Sigma  M
\rightarrow
0,
\]
which, in homological degree $s$, is the suspension of
\[
 0
\rightarrow
\Sigma^{-1} \rfunc_s (\Sigma ^s M)
\rightarrow
\rfunc _s (\Sigma^{s-1} M)
\stackrel{\rho_s}{\rightarrow}
\Sigma^{-2} \Phi \Sigma \rfunc_{s-1} (\Sigma^{s-1}M)
\rightarrow
0.
\]
\end{theorem}

\begin{proof}
It is sufficient to show that the short exact sequences of Proposition
\ref{prop:rho_s_fund_exact} induce morphisms of chain complexes.
By construction, the monomorphisms are compatible  with the differential, hence
it suffices to show that the morphisms $\rho_s$ induce a morphism of chain
complexes.

The cases $s \leq 1$ are immediate, hence fix $s \geq 2 $ and set  $N :=
\Sigma^{s-1}M$. There is a natural commutative diagram in
$\amod$:
\[
 \xymatrix{
\rfunc_s N
\ar@{^(->}[r]
\ar@{^(->}[d]
&
\rfunc_1  \rfunc_{s-1} N
\ar@{^(->}[d]
\ar[r]^{\rho_1}
&
\Sigma^{-2} \Phi \Sigma \rfunc_{s-1} N
\ar@{^(->}[d]
\\
\rfunc_{s-1}\rfunc_1 N
\ar@{^(->}[r]
\ar[d]
&
\rfunc_1 \rfunc_{s-2} \rfunc_1 N
\ar[r]^{\rho_1}
\ar[d]
&
\Sigma^{-2} \Phi \Sigma \rfunc_{s-2}\rfunc_1 N
\ar[d]
\\
\rfunc_{s-1} \Sigma^{-1} N
\ar@{^(->}[r]
&
\rfunc _1 \rfunc_{s-2} \Sigma^{-1} N 
\ar[r]_{\rho_1}
&
\Sigma^{-2} \Phi \Sigma \rfunc_{s-2}\Sigma^{-1}N,
}
\]
in which the lower vertical arrows are induced by $\rfunc_1 N\stackrel{d}{
\rightarrow} \Sigma^{-1} N$ and the upper
part of the
 diagram corresponds to embedding the functor $\rfunc_s$ in composite functors
(see Corollary
\ref{cor:rfunc_composites}). The commutativity of the right hand side of the
diagram follows from the naturality of
$\rho_1$.

The top and bottom horizontal composites correspond respectively to $\rho_s $
and $\rho_{s-1}$, whereas the the left
and right hand vertical composites are (up to suspension) the differentials in
the respective chain complexes. The result follows.
\end{proof}

\section{Derived functors of destabilization}
\label{sect:proof}

 The proof of Theorem \ref{thm:deriv_destab}, the main result of the paper,  is
analogous to the approach taken by Singer \cite{singer_loops_II} to the study of
the derived functors of iterated loop functors.

\subsection{First properties of the chain complex $\dcx_\bullet M$}

\begin{nota}
For $M \in \obj \amod$ and $s \in \nat$, write $\dhom_s M$ for $H_s
(\dcx_\bullet M)$.
\end{nota}

The following two lemmas are  straightforward:

\begin{lemma}
\label{lem:dhom0}
 For $s \in \nat$, $\dhom_s$ defines an additive  functor $\amod \rightarrow
\amod$ and  a short exact
sequence $0 \rightarrow K \rightarrow M \rightarrow Q \rightarrow 0$ of $\amod$
induces a long exact sequence in $\amod$:
\[
 \ldots \rightarrow \dhom_s K \rightarrow \dhom_s M \rightarrow \dhom_s Q
\rightarrow \dhom_{s-1} K \rightarrow \ldots \ \ .
\]
\end{lemma}

\begin{lemma}
\label{lem:les_dhom}
For $M \in \obj \amod$, the short exact sequence 
\[
 0
\rightarrow
\Sigma ^{-1}\dcx_\bullet \Sigma M
\rightarrow
\dcx _\bullet M
\stackrel{\rho_\bullet}{\rightarrow}
\Sigma^{-1} \Phi \dcx_{\bullet - 1}\Sigma  M
\rightarrow
0
\]
of Theorem \ref{thm:ses_complexes} induces a natural  long exact sequence in
$\amod$:
\[
 \ldots
\rightarrow
\Sigma^{-1} \dhom_s \Sigma M
\rightarrow
\dhom_s M
\rightarrow
\Sigma^{-1} \Phi \dhom_{s-1} \Sigma M
\stackrel{\lambda_{s-1}}{\rightarrow}
\Sigma^{-1} \dhom_{s-1} \Sigma M
\rightarrow \ldots  \ .
\]
\end{lemma}

The $\kappa_s$-connectivity of $\rfunc_s$ (see Proposition
\ref{prop:rfunc_s_properties}) implies the following:

\begin{lemma}
\label{lem:dhom_strong_conn}
 For $s \in \nat$ and $M \in \obj \amod$, 
 $
 |\dhom_s M | \geq 1 + p^s (|M| + s -1).
$
\end{lemma}

By construction of the chain complex and Proposition \ref{prop:coker_dM_destab}:

\begin{proposition}
\label{prop:D0_destab}
 For $M \in \obj \amod$,   $\dhom_0 M \cong
\destab M$.
\end{proposition}

\subsection{The main results}

The following provides the input to the  delta-functor type argument used
to prove Theorem \ref{thm:deriv_destab}:

\begin{proposition}
\label{prop:delta_functor}
For $t \in \zed$ and  $0< s \in \nat $, $\dhom_s (\Sigma^t \cala) =0$.
\end{proposition}

\begin{proof}
The proof is by induction on $s$, starting with $s=1$. In low degrees, the long
exact sequence of Lemma \ref{lem:les_dhom} has the form
\[
\Sigma^{-1}  \dhom_1 \Sigma M
\rightarrow
\dhom_ 1 M
\rightarrow
\Sigma^{-1} \Phi \destab \Sigma M
\stackrel{\Sigma^{-1}\lambda}{\rightarrow}
\Sigma^{-1} \destab \Sigma M,
\]
using Proposition \ref{prop:D0_destab}  to identify $\dhom_0$ with $\destab$ and
Proposition
\ref{prop:connecting_lambda} to
identify the connecting morphism. For $M= \Sigma^t \cala$,  the  morphism
$\Sigma^{-1} \lambda$ is the desuspension of
\[
 \lambda : \Phi F (t+1) \rightarrow F (t+1),
\]
since $\destab \Sigma (\Sigma^t \cala)$ is the free unstable module $F(t+1)$.
This morphism is injective (see \cite[Lemme 3.1.1]{zarati_these} or
\cite{li_these}), hence
$
 \Sigma^{-1} \dhom_1 (\Sigma^{t+1} \cala ) \twoheadrightarrow \dhom_1 (\Sigma^t
\cala), 
$
is surjective. It follows, using $\kappa_1$-connectivity (cf. Lemma
\ref{lem:dhom_strong_conn}), that $\dhom_1 (\Sigma^t \cala)=0$,  $ \forall t$.

The inductive step uses a similar argument: by induction, we may suppose that
$\dhom_{s-1} \Sigma^t \cala$ is zero  $\forall t$. Hence, by  Lemma
\ref{lem:les_dhom}, 
$
 \Sigma^{-1} \dhom_s \Sigma^{t+1} \cala \twoheadrightarrow \dhom_s \Sigma^t
\cala 
$
is surjective 
 $\forall t$. Again, it follows from the connectivity result Lemma
\ref{lem:dhom_strong_conn} that $\dhom_s (\Sigma^t \cala)=0$, for all
integers $t$.
\end{proof}

The main result of the paper follows, as for the proof of \cite[Theorem
6.6]{singer_loops_II}.

\begin{theorem}
\label{thm:deriv_destab}
 For $M \in \obj \amod$ and $s \in \nat$, there is a natural isomorphism
\[
 \dhom_s M \cong \destab_s M.
\]
\end{theorem}

\begin{remark}
 The theorem shows, in particular, that the homology $H_s (\dcx_\bullet M) $ is
an unstable module, which is not
transparent from the construction.
\end{remark}

\begin{corollary}
\label{cor:stronger_connectivity}
 For $s \in \nat$  and $M \in \obj \amod$, 
$
 |\destab_s M | \geq 1 + p^s (|M| + s -1).
$
\end{corollary}

Theorem \ref{thm:deriv_destab}  recovers the main results of
\cite{zarati_these}:

\begin{corollary}
 \cite[Théorème 2.5]{zarati_these}
\label{cor:unstable_rs}
For $M \in \obj \unst$ and $s \in \nat$, there is a natural isomorphism
$
 \destab_s (\Sigma^{1-s}M ) \cong \Sigma \rfunc_s  M.
$
\end{corollary}

\begin{proof}
 The relevant portion of the chain complex $\dcx_\bullet \Sigma^{1-s} M $ is
\[
 \Sigma \rfunc_{s+1} \Sigma M
\rightarrow
\Sigma \rfunc _s M
\rightarrow
\Sigma
\rfunc _{s-1}  \Sigma^{-1} M;
\]
the fact that $M$ is unstable implies that the two differentials are trivial.
\end{proof}

Recall that the loop functor $\Omega : \unst \rightarrow \unst $ is the left
adjoint to  suspension  $\Sigma$  and $\Omega_1$
is the unique non-trivial higher left derived functor of $\Omega$. There is a 
short exact sequence for calculating the composite of derived functors of
destabilization with desuspension (cf. \cite{zarati_these}),
 for $M \in \obj  \unst$ and $s \in \nat$:
\[
0
\rightarrow
 \Omega \destab_s M
\rightarrow
\destab_s \Sigma^{-1} M
\rightarrow
\Omega_1 \destab_{s-1} M
\rightarrow
0.
\]

\begin{corollary}
\label{cor:destab_s_Sigma_-s}
(Cf. \cite[Section 4.8]{zarati_these}.)
For $M \in \obj \unst$  and $0< s \in \nat$, there is a natural short exact
sequence of unstable
modules
\[
 0
\rightarrow
\rfunc_s M
\rightarrow
\destab_s \Sigma ^{-s}M
\rightarrow
\Omega_1 \destab_{s-1} \Sigma^{1-s} M
\rightarrow
0.
\]
\end{corollary}

\begin{remark}
As observed by Zarati in \cite[Remarque 4.8]{zarati_these}, the morphism
$\rfunc_s M
\rightarrow \destab_s \Sigma ^{-s}M$ is not in general an isomorphism; for
example,  $\rfunc_s \field
\rightarrow \destab_s \Sigma^{-s} \field$ is not an isomorphism for $s \gg 0$.
\end{remark}

\section{Module structures}
\label{sect:module}

For  $s\in \nat$, $K_s^+$ is a sub-unstable algebra of $H^* (BV_s)$ and, for any
$\cala$-module $N$, $\rfunc_s N$ is naturally an object of $K_s^+\dash \amod$.
The purpose of this section is to show
that some of this structure passes to the derived functors $\destab_s$.

\subsection{Linearity results}
\label{subsect:linearity}

As in Section \ref{subsect:d2_vanish}, the work of Hung and Sum  can be used to
analyse the chain complex $\dcx_\bullet M$, generalizing Lemma
\ref{lem:compatibility_diff} as follows. The multiplicative coproduct
$\psi_{s-1, 1}: \boldDelta_{s} \rightarrow \boldDelta_{s-1} \otimes
\boldDelta_1$ of \cite[Section 3]{hung_sum}, restricts to a coproduct
$\psi_{s-1,1} : \boldGamma_{s} \rightarrow \boldGamma_{s-1} \otimes
\boldGamma_1$. By the identification of the algebra $\boldGamma_s$ (cf.
Proposition \ref{prop:Gamma_Delta}), this is determined by the following special
case of \cite[Proposition 3.3]{hung_sum}.

\begin{lemma}
\label{lem:calculate_psi_s-1_1}
 For $0< s\in \nat$:
 \begin{eqnarray*}
\psi_{s-1,1} Q_{s,j} &=&
\left\{
\begin{array}{ll}
 Q_{s-1, 0}^p \otimes Q_{1,0} & j = 0 \\
Q_{s-1,0}^{p-1}Q_{s-1,j} \otimes Q_{1,0} + Q_{s-1, j-1}^p \otimes 1 &  j>0; \\
\end{array}
\right.
\\
\psi_{s-1,1} R_{s,j} &=&
\left\{
\begin{array}{ll}
Q_{s-1,0}^{p-1}R_{s-1,j} \otimes Q_{1,0} + Q_{s-1, 0}^{p-1}Q_{s-1,j} \otimes
R_{1,0} & 0\leq j < s-1\\
Q_{s-1, 0}^{p-1} \otimes R_{1,0}&  j =s-1.
\end{array}
\right.
\end{eqnarray*}
\end{lemma}

\begin{definition}
\label{def:partial_s}
\cite[Definition 3.5 and Proposition 3.6(ii)]{hung_sum}
	For $0< s\in \nat$ , let $\partial_s :\boldGamma_s \rightarrow 
\boldGamma_{s-1} \otimes \Sigma^{-1} \field_p$
denote the composite
\[
	\boldGamma_s
\stackrel{\psi_{s-1,1}}{\rightarrow}
\boldGamma_{s-1} \otimes \boldGamma_1
\stackrel{\boldGamma_{s-1} \otimes \partial}{\rightarrow}
\boldGamma_{s-1} \otimes \Sigma^{-1} \field_p.
\]
\end{definition}

\begin{proposition}
	\label{prop:diff_d_partial}
 For $M \in \obj \amod$ and $0< s\in \nat$, there is a natural commutative
diagram
\[
 \xymatrix{
\rfunc_s \Sigma M
\ar[r]^{d_s}
\ar@{^(->}[d]
&
 \rfunc_{s-1} M
\ar@{^(->}[d]
\\
\boldGamma_s \largetensor \Sigma M
\ar[r]_{\partial_s \largetensor \Sigma M}
&
 \boldGamma_{s-1} \largetensor M,
}
\]
in which the vertical morphisms are the canonical inclusions.
\end{proposition}

\begin{proof}
The result follows from Lemma \ref{lem:compatibility_diff} (cf. also Remark
\ref{rem:Gamma_plus}).
\end{proof}

The key to the linearity results is the following Lemma.

\begin{lemma}
\label{lem:Q^u_kill_fractions}
For $M \in \obj \unst $,  $t \in \nat$  and  $u:= \big[\frac{t+1}{2}\big]$, the
submodule $Q_{s,0} ^u \rfunc_s (\Sigma^{-t} M) $ of $\rfunc_s (\Sigma^{-t}M)
\subset \boldGamma_s \otimes \Sigma^{-t} M$
is contained in $H^* (BV_s) ^{GL_s} \otimes \Sigma^{-t}M\subset \boldGamma_s
\otimes \Sigma^{-t} M$.
\end{lemma}

\begin{proof}
 By choice of $u$, $\Sigma^{2 u } \Sigma^{-t} M$ is unstable and hence $\rfunc_s
\Sigma^{2 u } \Sigma^{-t} M$ is a
submodule of $H^* (BV_s) ^{GL_s} \otimes \Sigma^{2 u } \Sigma^{-t} M$.

There is a natural isomorphism
$
 \rfunc_s \Sigma^{2 u } \Sigma^{-t} M \cong
\Sigma^{2u} Q_{s,0}^{u} \rfunc_s (\Sigma^{-t}M),
$
 considered as a submodule of $\Sigma^{2u}\rfunc_s (\Sigma^{-t}M)$, by
Proposition \ref{prop:rfuncs_suspension}. The
result follows.
\end{proof}

Consider the differential $
 d_s : \Sigma \rfunc_s \Sigma^{s-1} N
\rightarrow
\Sigma \rfunc_{s-1} \Sigma^{s-2} N
$ of $\dcx _\bullet N$. The first term is  a module over $\field_p[Q_{s,i}]$ and
the second over $\field_p
[Q_{s-1,j}]$, hence over $\field_p[Q_{s,i}]$, via $\phi_s$ of Lemma
\ref{lem:dickson_phi_s}.

In general, the differential is not $\field_p [Q_{s,i}]$-linear; however it
becomes linear when $N$ is an iterated desuspension
of an unstable module, after restricting to subalgebras of the form
$\Phi^k \field_p [Q_{s,i}]$, for suitably large $k$.

\begin{proposition}
\label{prop:diff_linearity}
For $M \in \obj \unst$,  $t\in \nat$, $u:= \big[\frac{t+1}{2}\big]$ and $w\in
\nat$
 such that $p^w \geq u$, the morphism
$ 
 d_s : \rfunc_s (\Sigma^{-t} M )
\rightarrow
\rfunc_{s-1} (\Sigma^{-(t+1) } M ).
$
is $\Phi^w \field_p[Q_{s,i}]$-linear.
\end{proposition}

\begin{proof}
By Proposition \ref{prop:diff_d_partial} there is
commutative diagram
\[
 \xymatrix{
\rfunc_s \Sigma^{-t} M
\ar[rr]^{d_s}
\ar@{^(->}[d]
&&
\rfunc_{s-1} \Sigma^{-(t+1)} M
\ar@{^(->}[d]
\\
\boldGamma_s \otimes \Sigma^{-t}M
\ar[r]_(.4){\psi_{s-1,1} \otimes 1}
&
\boldGamma_{s-1} \otimes \boldGamma_1 \otimes \Sigma^{-t}M
\ar[r]_{1 \otimes \partial \otimes 1 }
&
\boldGamma_{s-1} \otimes \Sigma^{-(t+1)}M.
}
\]
(Here the instability of $M$ implies that  large tensor products are
unnecessary.)

The vertical inclusions are morphisms of  $\field_p [Q_{s,i}]$-modules and
$\field_p [Q_{s-1,j}]$-modules respectively; the latter can be
considered as a morphism of $\field_p
[Q_{s,i}]$-modules as above, via the morphism
$\phi_s$ of Lemma \ref{lem:dickson_phi_s}. It suffices to show that the
composite is $\Phi^w \field_p[Q_{s,i}]$-linear.

Lemma \ref{lem:Q^u_kill_fractions} and Lemma \ref{lem:calculate_psi_s-1_1} imply
that the image of
$\rfunc_s \Sigma^{-t}M $ in $\boldGamma_{s-1} \otimes \boldGamma_1 \otimes
\Sigma^{-t} M$ lies in the submodule
$$(Q_{s-1,0}^p \otimes Q_{1,0} ) ^{-u}\Big(H^* (BV_{s-1} )^{GL_{s-1}} \otimes
H^* (BV_1)^{GL_1} \otimes \Sigma^{-t}
M\Big).$$

After multiplying by an element of $H^* (BV_{s-1} )^{GL_{s-1}} \otimes H^*
(BV_1)^{GL_1}$ of the form $\alpha \otimes
Q_{1,0} ^n$, where $n \geq u$, this element lies in $\boldGamma_{s-1}\otimes H^*
(BV_1)^{GL_1} \otimes \Sigma^{-t}
M$, which is contained in the kernel of the morphism induced by $\partial_1$.

By Lemma \ref{lem:calculate_psi_s-1_1},
$$\psi_{s-1,1} Q_{s,j}^{p^w} =
\left\{
\begin{array}{ll}
 Q_{s-1, 0}^{p^{w+1}} \otimes Q_{1,0}^{p^w} & j = 0 \\
Q_{s-1,0}^{(p-1)p^w}Q_{s-1,j}^{p^w} \otimes Q_{1,0}^{p^w} + Q_{s-1,
j-1}^{p^{w+1}} \otimes 1 &  j>0. \\
\end{array}
\right.
$$
The hypothesis $p^w \geq u$ allows the previous remark to be applied, so that
the terms involving
$Q_{1,0}^{p^w}$ can be discarded. The result follows.
\end{proof}

\subsection{Module structures on derived functors of destabilization}

\begin{theorem}
\label{thm:module_Ds}
For $M \in \obj \unst$ and  $s, t, w\in \nat$ such that $p^w \geq \big[ 
\frac{t-s+1}{2} \big]$, 
 $\destab_s (\Sigma^{-t} M)$ has a natural $\Phi^{w+1}\field_p
[Q_{s,i}]$-structure in  $\unst$.

If $t \leq s$,  then $\destab_s (\Sigma^{-t}M ) $ has a natural $ \field_p
[Q_{s,i}]$-module
structure in $\unst$.
\end{theorem}

\begin{proof}
By Theorem \ref{thm:deriv_destab}, $\destab_s \Sigma^{-t} M $ is calculated as
the homology of:
\[
 \Sigma \rfunc_{s+1} (\Sigma^{s-t}M)
\stackrel{d_{s+1}}{\rightarrow}
 \Sigma \rfunc_{s} (\Sigma^{s-t-1}M)
\stackrel{d_{s}}{\rightarrow}
 \Sigma \rfunc_{s-1} (\Sigma^{s-t-2}M).
\]
The morphism $d_s$ is $\Phi^{w} \field_p [Q_{s,i}]$-linear,
by Proposition \ref{prop:diff_linearity},
 hence the kernel has a natural $\Phi^{w} \field_p [Q_{s,i}]$-module structure
in $\amod$, which
restricts to a natural $\Phi^{w+1} \field_p [Q_{s,i}]$-module
structure in $\amod$.

Similarly, the morphism $d_{s+1}$ is $\Phi^{w} \field_p [Q_{s+1,l}]$-linear;
hence the image is
naturally a sub $\Phi^{w} \field_p [Q_{s+1,l}]$-module of 
$\Sigma \rfunc_{s} (\Sigma^{s-t-1}M)$ in
$\amod$,  where the $\field_p [Q_{s+1,l}]$-structure on
$\Sigma \rfunc_{s} (\Sigma^{s-t-1}M)$ is induced by
restriction along $\phi_{s+1} : \field_p [Q_{s+1,l}]
\rightarrow  \field_p [Q_{s,i}]$, which surjects onto the subalgebra $\Phi
\field_p [Q_{s,i}]$ by Lemma \ref{lem:dickson_phi_s}. It follows that the image
of
$d_{s+1}$ is a sub $\Phi^{w+1} \field_p
[Q_{s,i}]$-module in $\amod$ and hence the homology has a
natural $\Phi^{w+1} \field_p
[Q_{s,i}]$-module structure, as required.

In the case $t\leq s$, one can take $w=0$; the differential $d_{s+1}$ is
trivial, hence the additional $\Phi$  is unnecessary.
\end{proof}

\begin{remark}
The case $t<s$ is immediate from Zarati's result (recovered here
as Corollary \ref{cor:unstable_rs}), which provides a $K_s^+$-module structure.
\end{remark}

\providecommand{\bysame}{\leavevmode\hbox to3em{\hrulefill}\thinspace}
\providecommand{\MR}{\relax\ifhmode\unskip\space\fi MR }
\providecommand{\MRhref}[2]{%
  \href{http://www.ams.org/mathscinet-getitem?mr=#1}{#2}
}
\providecommand{\href}[2]{#2}

\end{document}